\documentclass[11pt]{article}
\usepackage{amssymb,amsmath,amsthm}
\usepackage{color}
\usepackage{marginnote}

\numberwithin{equation}{section}
\newtheorem{thm}{Theorem}[section]
\newtheorem{df}[thm]{Definition}
\newtheorem{prop}[thm]{Proposition}
\newtheorem{lem}[thm]{Lemma}
\newtheorem{rem}[thm]{Remark}

\newtheorem{cor}[thm]{Corollary}
\let\oldproofname=\proofname
\renewcommand{\proofname}{\rm\bf{\oldproofname}}

\newcommand{\N}{\mathbb{N}}

\newcommand{\R}{\mathbb{R}}

\newcommand{\cA}{\mathcal{A}}
\newcommand{\cC}{\mathcal{C}}
\newcommand{\cD}{\mathcal{D}}
\newcommand{\cE}{\mathcal{E}}
\newcommand{\cF}{\mathcal{F}}
\newcommand{\cG}{\mathcal{G}}
\newcommand{\cH}{\mathcal{H}}
\newcommand{\cI}{\mathcal{I}}
\newcommand{\cL}{\mathcal{L}}
\newcommand{\cM}{\mathcal{M}}
\newcommand{\cO}{\mathcal{O}}
\newcommand{\cR}{\mathcal{R}}

\newcommand{\dd}{\,{\rm d}}
\newcommand{\D}{{\rm d}}
\renewcommand{\div}{\mathop{\mathrm{div}}\nolimits}
\newcommand{\curl}{\mathop{\mathrm{curl}}}

\newcommand{\supp}{\mathop{\mathrm{supp}}}

\newcommand{\logp}{\mathrm{log}_+}
\newcommand{\1}{\mathbf{1}}
\newcommand{\weakto}{\rightharpoonup}
\newcommand{\tv}{\mathrm{tv}}

\newcommand{\BMO}{\mathrm{BMO}}
\newcommand{\BS}{\mathrm{BS}}
\newcommand{\QED}{\mbox{}\hfill$\Box$}

\begin{document}

\title{Uniqueness of axisymmetric viscous flows\\
originating from circular vortex filaments}

\author{
\null\\
{\bf Thierry Gallay}\\
Institut Fourier\\
Universit\'e Grenoble Alpes et CNRS\\
100 rue des Maths\\
38610 Gi\`eres, France\\
{\tt Thierry.Gallay@ujf-grenoble.fr}
\and
\\
{\bf Vladim\'ir \v{S}ver\'ak}\\
School of Mathematics\\
University of Minnesota\\
127 Vincent Hall,
206 Church St.\thinspace SE\\
Minneapolis, MN 55455, USA\\
{\tt sverak@math.umn.edu}}

\date{September 6, 2016}

\maketitle

\begin{abstract}
The incompressible Navier-Stokes equations in $\R^3$ are shown to
admit a unique axisymmetric solution without swirl if the initial
vorticity is a circular vortex filament with arbitrarily large
circulation Reynolds number. The emphasis is on uniqueness, as
existence has already been established in~\cite{FS}. The main
difficulty which has to be overcome is that the nonlinear regime for
such flows is outside of applicability of standard perturbation
theory, even for short times. The solutions we consider are archetypal
examples of viscous vortex rings, and can be thought of as
axisymmetric analogues of the self-similar Lamb-Oseen vortices in
two-dimensional flows. Our method provides the leading term in a 
fixed-viscosity short-time asymptotic expansion of the solution, and 
may in principle be extended so as to give a rigorous justification, 
in the axisymmetric situation, of higher-order formal asymptotic 
expansions that can be found in the literature \cite{Callegari-Ting}.
\end{abstract}

\section{Introduction}\label{sec1}

In three-dimensional ideal fluids, a {\em vortex ring} is an
axisymmetric flow with the property that the vorticity is entirely
concentrated in a solid torus, which moves with constant speed along
the symmetry axis. The vortex lines form large circles that fill the
torus, whereas fluid particles spin around the vortex core within
perpendicular cross sections. If $\bar r, r$ denote the major and
minor radii of the torus, respectively, and if $\Gamma$ is the flux of
the vorticity vector through any cross section, the ``local induction
approximation'' gives the following expression for the translation
speed along the axis
\begin{equation}\label{V-LIA}
  V \,=\, \frac{\Gamma}{4\pi\bar r}\Bigl(\log\frac{1}{\epsilon} +
  \cO(1)\Bigr)\,,
\end{equation}
which is valid in the asymptotic regime where the aspect ratio
$\epsilon = r/\bar r$ is small. For the three-dimensional Euler equations,
existence of large families of uniformly translating vortex ring
solutions has been obtained using fixed point methods \cite{Fr} or
variational techniques \cite{AS,FB,FT}, and formula \eqref{V-LIA} has
been rigorously justified when $\epsilon \ll 1$ \cite{Fr,FT}.  In
addition, for general initial data that are close enough to a vortex
ring with small aspect ratio, it is known that the solution evolves
in such a way that the vorticity remains sharply concentrated,
for a relatively long time, near a vortex ring whose speed is
given by \eqref{V-LIA}, see \cite{BCM}.

The situation is quite different for viscous fluids, in which
uniformly translating vortex rings cannot exist because all localized
structures are eventually spread out by diffusion. In that case,
however, it is quite natural to consider the initial value problem
with a {\em vortex filament} as initial data, namely a vortex ring
with infinitesimal cross section and yet nonzero circulation $\Gamma$,
so that the initial vorticity is a measure supported by a circle of
radius $\bar r$. It is then expected that the solution at time $t > 0$
will be close to a vortex ring with Gaussian vorticity profile and
minor radius $r = \sqrt{\nu t}$, where $\nu$ is the kinematic
viscosity. Moreover, this vortex will move along its symmetry axis at
a speed given by \eqref{V-LIA}, as long as the time-dependent aspect
ratio $\epsilon = \sqrt{\nu t}/\bar r$ is sufficiently small.

Justifying these heuristic considerations requires some work. For
singular initial data such as vortex filaments, the best available
results on the Cauchy problem for the three-dimensional Navier-Stokes
equations provide existence of a (unique and global) solution only if
the circulation parameter $\Gamma$ is small enough compared to
viscosity, see \cite{GM,KT}. For larger values of $\Gamma/\nu$,
existence of a (global) axisymmetric solution without swirl has been
recently obtained by H.~Feng and the second author \cite{FS}, using
approximation techniques that do not give any information about
uniqueness, even within the axisymmetric class. In this paper, our
main purpose is to fill this gap and to prove that, if one starts from
a circular vortex filament with arbitrary strength $\Gamma$, the
Navier-Stokes equations have a {\em unique} axisymmetric solution
without swirl, which is global and smooth for positive times.  This
axisymmetric solution is the archetype of a viscous vortex ring, just
as the two-dimensional Lamb-Oseen solution is the archetype of a
viscous columnar vortex \cite{GW2}. Our approach is constructive and
allows us to determine the leading term in the short-time asymptotic
expansion of the vortex ring for a fixed viscosity. In principle,
performing the calculations to higher orders in the spirit of
Callegari and Ting's paper~\cite{Callegari-Ting}, one should be able
to obtain more precise approximations of the solution that remain
valid as long as the aspect ratio $\epsilon = \sqrt{\nu t}/ \bar r$ is
small enough. In particular, computing the next order after the
leading term, we should recover the asymptotic formula \eqref{V-LIA}
for the translation speed if $|\Gamma|/\nu \gg 1$. We leave this
extension for future work.

To state our results in a more precise way, we start from the
Navier-Stokes equations
\begin{equation}\label{NS3D}
  \partial_t u + (u\cdot \nabla)u \,=\, \nu\Delta u - \frac{1}{\rho}
  \nabla p\,, \qquad \div u \,=\, 0\,,
\end{equation}
in the whole space $\R^3$, where $u = u(x,t) \in \R^3$ denotes
the velocity field and $p = p(x,t) \in \R$ is the internal
pressure. Both the kinematic viscosity $\nu > 0$ and the
fluid density $\rho > 0$ are assumed to be constant.
We restrict ourselves to {\em axisymmetric solutions without swirl}
for which the velocity field $u$ and the vorticity $\omega =
\curl u$ have the particular form:
\begin{equation}\label{uomaxi}
  u(x,t) \,=\, u_r(r,z,t) e_r + u_z(r,z,t) e_z\,, \qquad
  \omega(r,z,t) \,=\, \omega_\theta(r,z,t)e_\theta\,.
\end{equation}
Here $(r,\theta,z)$ are the usual cylindrical coordinates in $\R^3$,
such that $x = (r\cos\theta,r\sin\theta,z)$ for any $x \in \R^3$,
and $e_r, e_\theta, e_z$ denote the unit vectors in the radial, toroidal,
and vertical directions, respectively. The axisymmetric vorticity
$\omega_\theta = \partial_z u_r - \partial_r u_z$ satisfies the
evolution equation
\begin{equation}\label{omeq}
  \partial_t \omega_\theta + u\cdot\nabla\omega_\theta -
  \frac{u_r}{r}\omega_\theta \,=\, \nu\Bigl(\Delta \omega_\theta
  - \frac{\omega_\theta}{r^2}\Bigr)\,,
\end{equation}
where $u\cdot\nabla = u_r \partial_r + u_z \partial_z$ and $\Delta
= \partial_r^2 + \frac1r \partial_r + \partial_z^2$ is the
axisymmetric Laplace operator in cylindrical coordinates.  The
velocity $u$ can be expressed in terms of the axisymmetric vorticity
$\omega_\theta$ by solving the linear elliptic system
\begin{equation}\label{diffBS}
  \partial_r u_r + \frac{1}{r} u_r + \partial_z u_z \,=\, 0\,,
  \qquad \partial_z u_r - \partial_r u_z \,=\, \omega_\theta\,,
\end{equation}
in the half-plane $\Omega = \{(r,z) \in \R^2\,|\, r > 0\,,~z \in
\R\}$. Boundary conditions for the quantities $u_r$, $u_z$, and
$\omega_\theta$ are prescribed by requiring that the vectorial
functions $u$, $\omega$ in~\eqref{uomaxi} be smooth across the
symmetry axis $r = 0$. One finds that the radial velocity $u_r$
and the axisymmetric vorticity $\omega_\theta$ should satisfy
the homogeneous Dirichlet condition on $\partial\Omega$, whereas
the vertical velocity $u_z$ satisfies the homogeneous Neumann
condition.

Since the pioneering work of Ukhovskii and Yudovitch \cite{UY}, and of
Ladyzhenskaya \cite{La}, it is well known that the axisymmetric
Navier-Stokes equations without swirl are globally well-posed for
velocities in (appropriate subspaces of) the energy class, see also
\cite{Ab,LMNP} for further results in this direction. In the recent
work \cite{GS}, the Cauchy problem for the vorticity equation
\eqref{omeq} is studied using scale invariant function spaces which
emphasize the analogy with the two-dimensional vorticity equation.
Following \cite{GS}, we equip the half-plane $\Omega$ with the
two-dimensional measure $\D r\dd z$, as opposed to the
three-dimensional measure $r\dd r\dd z$ which appears more naturally
in cylindrical coordinates. In particular, for any $p \in [1,\infty)$,
we denote by $L^p(\Omega)$ the space of measurable functions
$\omega_\theta : \Omega \to \R$ such that
\[
  \|\omega_\theta\|_{L^p(\Omega)} \,:=\, \left(\int_\Omega
  |\omega_\theta(r,z)|^p \dd r\dd z\right)^{1/p} \,<\, \infty\,.
\]
As usual, the limiting space $L^\infty(\Omega)$ is equipped with
the essential supremum norm. We also denote by $\cM(\Omega)$
the set of all real-valued finite regular measures on $\Omega$,
equipped with the total variation norm
\[
  \|\mu\|_\tv \,=\, \sup\left\{ \int_\Omega \phi \dd\mu \,\Big|\,
   \phi \in C_0(\Omega)\,,~ \|\phi\|_{L^\infty(\Omega)} \le 1
   \right\}\,,
\]
where $C_0(\Omega)$ denotes the set of all real-valued continuous
functions on $\Omega$ that vanish at infinity and on the boundary
$\partial\Omega$. Clearly $L^1(\Omega)$ is a closed subspace
of $\cM(\Omega)$, and $\|\mu\|_\tv = \|\omega_\theta\|_{L^1(\Omega)}$
if $\mu = \omega_\theta\dd r\dd z$ for some $\omega_\theta
\in L^1(\Omega)$.

As is proved in \cite[Theorem~1.3]{GS}, the Cauchy problem for the
axisymmetric vorticity equation \eqref{omeq} is globally well-posed if
the initial vorticity $\mu \in \cM(\Omega)$ is a finite measure whose
atomic part $\mu_{pp}$ satisfies $\|\mu_{pp}\|_\tv \le C_0 \nu$, where
$C_0 > 0$ is a universal constant. We are especially interested here
in the particular situation where $\mu = \Gamma\,\delta_{(\bar r,
\bar z)}$, which corresponds to a circular vortex filament of
strength $\Gamma \in \R$ and radius $\bar r > 0$, centered at the
origin in the affine plane $x_3 = \bar z \in \R$. In that case, we
have $\|\mu_{pp}\|_\tv = \|\mu\|_\tv = |\Gamma|$, so that the results
of \cite{GS} assert the existence of a unique global solution if
$|\Gamma| \le C_0 \nu$. On the other hand, for arbitrary values of the
circulation parameter $\Gamma$, existence of a global solution to
\eqref{omeq} was recently obtained by H.~Feng and the second author
\cite{FS}, using approximation techniques which however do not give
any information about uniqueness.

With this perspective in mind, our main result can now be stated
as follows:

\begin{thm}\label{main}
Fix $\Gamma \in \R$, $\bar r > 0$, $\bar z \in \R$, and $\nu > 0$.
Then the axisymmetric vorticity equation~\eqref{omeq} has a unique
global mild solution $\omega_\theta \in C^0((0,\infty),
L^1(\Omega) \cap L^\infty(\Omega))$ such that
\begin{equation}\label{omcond}
  \sup_{t > 0} \|\omega_\theta(t)\|_{L^1(\Omega)} \,<\, \infty\,,
  \qquad \hbox{and}\quad \omega_\theta(t)\dd r\dd z \weakto
  \Gamma \,\delta_{(\bar r,\bar z)} \quad \hbox{as }t \to 0\,.
\end{equation}
In addition, there exists a constant $C_1 > 0$, depending
only on the ratio $|\Gamma|/\nu$, such that the following
estimate holds:
\begin{equation}\label{short-time}
  \int_\Omega \,\Bigl|\omega_\theta(r,z,t) - \frac{\Gamma}{4\pi\nu t}
  \,e^{-\frac{(r-\bar r)^2+(z-\bar z)^2}{4\nu t}}\Bigr|\dd r \dd z
  \,\le\, C_1 \,|\Gamma|\,\frac{\sqrt{\nu t}}{\bar r}\,\log
  \frac{\bar r}{\sqrt{\nu t}}\,,
\end{equation}
as long as $\sqrt{\nu t} \le \bar r/2$.
\end{thm}

At this point a few comments are in order.

\medskip{\bf 1.} Theorem~\ref{main} can be seen as the
axisymmetric counterpart of Proposition~1.3 in \cite{GW2}, which
characterizes the Lamb-Oseen vortices among all two-dimensional
viscous flows. We recall that a mild to solution to
\eqref{omeq} is a solution of the associated integral
equation, see Definition~\ref{def-mild} below.

\smallskip{\bf 2.} As was already mentioned, existence of a
global mild solution to \eqref{omeq} satisfying \eqref{omcond} was
established in \cite{FS}. Uniqueness is thus the main new assertion in
Theorem~\ref{main}, together with the small time asymptotic expansion
\eqref{short-time}. It should be mentioned, however, that the techniques
developed in Section~\ref{sec4}, when properly adapted, can provide 
existence of a solution to \eqref{omeq} satisfying \eqref{omcond}, 
using a standard fixed point argument which also gives uniqueness 
in a restricted class.

\smallskip{\bf 3.} Assumptions \eqref{omcond} are the weakest ones
under which the conclusions of Theorem~\ref{main} are expected to
hold. Indeed, we recall that the $L^1$ norm of any solution to
\eqref{omeq} is a nonincreasing function of time, see
\cite[Lemma~5.1]{GS}, and it follows from \eqref{short-time} that
$\|\omega_\theta(t)\|_{L^1(\Omega)} \to |\Gamma|$ as $t \to 0$, hence
the first condition in \eqref{omcond} is clearly necessary. The second
hypothesis states that $\omega_\theta(t)$ converges to $\Gamma
\,\delta_{(\bar r,\bar z)}$ as $t \to 0$ in the weak-star topology of
$\cM(\Omega)$, which is usually referred to as the ``weak convergence
of measures''.  But since $\omega_\theta(t)$ is uniformly bounded in
$L^1(\Omega)$, it is equivalent to suppose that convergence holds in
the sense of distributions on $\Omega$, and this is arguably the
weakest way to specify the initial data.

\smallskip{\bf 4.} The short time estimate \eqref{short-time} is sharp
in the sense that the right-hand side cannot be replaced by $C_1
|\Gamma| \epsilon$, where $\epsilon = \sqrt{\nu t}/\bar r$ is the
aspect ratio at time $t$. This is because, in \eqref{short-time}, we
compare the solution $\omega_\theta(t)$ to a viscous vortex ring
located at a fixed point $(\bar r,\bar z)$ in cylindrical coordinates,
whereas we know that any vortex ring should move in the vertical
direction at a speed given approximately by \eqref{V-LIA}. As a
matter of fact, it is not difficult to verify that, if we replace
in \eqref{short-time} the fixed vertical coordinate $\bar z$ by
\[
  \bar z(t) \,=\, \bar z + \frac{\Gamma t}{4\pi \bar r}
  \,\log \frac{\bar r}{\sqrt{\nu t}}\,,
\]
then estimate \eqref{short-time} holds without the logarithmic term in
the right-hand side. More generally, the Gaussian vorticity profile in
\eqref{short-time} is only the first term in an asymptotic expansion
of the solution $\omega_\theta(t)$ which, in principle, can be computed
to arbitrary order in $\epsilon$.

\smallskip{\bf 5.} In estimate \eqref{short-time} convergence is
expressed in the $L^1$ norm for simplicity, but in the proof we use a
weighted $L^2$ norm in self-similar variables, which is considerably
stronger and also implies approximation results for the velocity field
associated with $\omega_\theta(t)$. On the other hand, we emphasize
that \eqref{short-time} is a short time result at fixed viscosity,
which cannot be used to describe the solution at fixed time $t > 0$ in
the vanishing viscosity limit $\nu \to 0$, because the constant $C_1$
in the right-hand side strongly depends on the ratio $|\Gamma|/\nu$.
Controlling the weakly viscous vortex ring over a fixed time interval
is different problem, which requires in particular constructing a
much more precise approximation of the solution $\omega_\theta(t)$.
We hope to address this interesting question in a future work.

\smallskip{\bf 6.}  It is worth emphasizing that the uniqueness
statement is proved only withing the class of axisymmetric
solutions. A natural question is whether uniqueness remains true among
all (reasonable) solutions of \eqref{NS3D} that approach the initial
vortex filament in a suitable sense as $t \to 0$. For instance, one
may assume that the velocity field $u(x,t)$ is smooth in $\R^3 \times
(0,\infty)$, and satisfies the scale-invariant estimates listed in
\eqref{deru-apriori} below. As for the associated vorticity
$\omega(x,t)$, one may suppose (motivated by~\cite{GM}, see also
Remark~\ref{BMO_rem} below) that a natural quantity such as
\begin{equation}\label{zz1}
  \sup_{x \in \R^3}\sup_{R > 0} \frac{1}{R} \int_{B_{x,R}} 
  |\omega(y,t)|\dd y
\end{equation} 
is uniformly bounded for $t > 0$, and that $\omega(t)$ approaches the
vortex filament in the sense of distributions as $t\to 0$. But even
under these strong assumptions, it seems to be a difficult open
problem to decide whether $u(x,t)$ has to be axisymmetric. It is
conceivable that the symmetry of the initial data can be broken and,
in addition to the axisymmetric solution, there is another solution
which is not axisymmetric. In fact, the same question already arises
for rectilinear vortices: the uniqueness problem when the initial
vorticity is a (vertical) straight vortex filament, considered within
the class of $x_3$-independent velocity fields of the form
$(u_1,u_2,0)$, is the same as the 2d uniqueness and has been solved
in~\cite{GG,GGL}, but uniqueness among reasonable classes of 3d vector
fields remains open.

The difficulties arise because the initial data do not belong to
functions spaces where perturbation theory gives existence and
uniqueness of local-in-time solutions for large data. Typical examples
of function spaces (for the velocity field) where large data can be
handled, locally in time, are the Lebesgue space $L^3(\R^3)$ or the
Besov space $\dot B^{-1+3/p}_{p,q}(\R^3)$ for $p\in(3,\infty)$ and
$q<\infty$. However, if the initial velocity belongs to
$\BMO^{-1}(\R^3)$, or if the vorticity lies in the Morrey space
$M^{3/2}(\R^3)$, the local-in-time well-posedness is completely
unclear for large data, even though the problem is globally well-posed
for small data \cite{GM,KT}. The vortex filaments considered in the
present paper are natural examples of initial data that belong to the
latter category of function spaces, but not to the former. Recently it
has been conjectured \cite{Jia-Sverak} that local-in-time
well-posedness, and indeed uniqueness, may fail already for initial
data $u_0$ that are compactly supported, smooth away from the origin,
and $(-1)$-homogeneous near the origin. A similar question can be 
raised a fortiori for vortex filaments, where the singularity of 
the initial data is not located at a single point but is spread over
a whole curve. 

\medskip Although the proof of Theorem~\ref{main} is the main purpose
of this paper, we establish on the way several auxiliary results that
have their own interest. As the existence part is already settled in
\cite{FS}, we concentrate on uniqueness and short time asymptotics. In
all what follows, we thus assume that we are given a mild solution
$\omega_\theta \in C^0((0,T),L^1(\Omega) \cap L^\infty(\Omega))$ of
\eqref{omeq} which is uniformly bounded in the space $L^1(\Omega)$. In
Section~\ref{sec2}, we recall some a priori estimates that were
obtained in \cite{GS}, and we show that $\omega_\theta(t)$ converges
weakly, as $t \to 0$, to some (uniquely defined) Radon measure $\mu
\in \cM(\Omega)$. Next, using recent results on linear parabolic
equations with singular divergence-free drifts \cite{SSSZ}, we prove
that the solutions of the adjoint equation to \eqref{omeq} are
continuous all the way to the initial time $t = 0$, even at the
symmetry axis $r = 0$. This nontrivial result allows us to deduce that
the family of measures $\omega_\theta(t)\dd r\dd z$ remains tight as
$t \to 0$, so that no mass can escape to infinity nor concentrate on
the symmetry axis.

In Section~\ref{sec3} we focus on the particular case where $\mu =
\Gamma \,\delta_{(\bar r,\bar z)}$ for some $(\bar r,\bar z) \in
\Omega$, assuming without loss of generality that $\Gamma > 0$.
We prove that the solution $\omega_\theta(t)$ is strictly positive
and satisfies, for any $\eta \in (0,1)$, the Gaussian bound
\begin{equation}\label{Gauss-intro}
  \omega_\theta(r,z,t) \,\le\, C\,\frac{\Gamma}{\nu t}\,
  \exp\Bigl(-\frac{1-\eta}{4\nu t}\bigl((r-\bar r)^2 +
  (z-\bar z)^2\bigr)\Bigr)\,,
\end{equation}
where $C > 0$ depends only on $\eta$ and on the ratio $\Gamma/\nu$. In
the two-dimensional case, estimates of the form \eqref{Gauss-intro}
were obtained by Osada \cite{Os}, see also \cite{CL}. Reproducing them
in the axisymmetric case is not straightforward, because the left-hand
side of \eqref{omeq} contains the zero order term $u_r\omega_\theta/r$
which is harmless only if one can prove that $\|u_r(t)/r\|_{L^\infty(\Omega)}$ 
is integrable in time. That property
does not follow from the scale invariant a priori estimates on the
solution, but we can show that it holds as soon as the support of the
initial measure $\mu$ is bounded away from the symmetry axis, which is
of course the case in our problem. Thus a minor modification of the
method presented in \cite{FS} allows us to establish the Aronson type
estimate \eqref{Gauss-intro}.

Section~\ref{sec4} is devoted to the actual proof of Theorem~\ref{main}.
To study the behavior of the solution near the location $(\bar r,\bar z)$
of the initial vortex filament, we introduce self-similar variables via
the transformation
\begin{equation}\label{f-intro}
  \omega_\theta(r,z,t) \,=\, \frac{\Gamma}{\nu t}\,f\Bigl(\frac{r-\bar r}
  {\sqrt{\nu t}}\,,\,\frac{z-\bar z}{\sqrt{\nu t}}\,,\,t\Bigr)\,.
\end{equation}
The rescaled vorticity $f(R,Z,t)$ defined by \eqref{f-intro} is
positive and, in view of \eqref{Gauss-intro}, bounded from above by a
Gaussian function. Using a compactness argument and a Liouville
theorem established in \cite{GW2}, we show that $f(t)$
converges as $t \to 0$ to the Gaussian $G$ defined by
\begin{equation}\label{G-intro}
  G(R,Z) \,=\, \frac{1}{4\pi} e^{-(R^2+Z^2)/4}\,, \qquad (R,Z) \in \R^2\,.
\end{equation}
Convergence holds in the weighted space $X = L^2(\R^2,G^{-1}\dd R\dd Z)$
which is continuously embedded in $L^1(\R^2)$, hence returning to
the original variables we deduce that the left-hand side of
\eqref{short-time} vanishes as $t \to 0$. We next use energy
estimates to show that the difference $\|f(t) - G\|_X$ is
$\cO(\epsilon|\log\epsilon|)$, where $\epsilon = \sqrt{\nu t}/\bar r$,
and this concludes the proof of \eqref{short-time}. Finally,
repeating the energy estimates for the difference of two
solutions satisfying the assumptions of Theorem~\ref{main},
we prove that $\|f_1(t) - f_2(t)\|_X = 0$ for sufficiently
small times, and invoking a well-posedness result from
\cite{GS} we conclude that the solutions coincide for all
times, which gives uniqueness.

The final Section~\ref{sec5} is an appendix were the proofs of
a few auxiliary results are collected for easy reference
in the text.

\medskip\noindent{\bf Acknowledgements.} This project started during
visits of the first named author to the University of Minnesota,
whose hospitality is gratefully acknowledged. Our research was
supported in part by grants DMS 1362467 and DMS 1159376 from the
National Science Foundation (V.S.), and by grant ``Dyficolti''
ANR-13-BS01-0003-01 from the French Ministry of Research (Th.G.).


\section{General properties of $L^1$-bounded solutions}\label{sec2}

In this section, we establish some preliminary results concerning mild
solutions of~\eqref{omeq} that are uniformly bounded in
$L^1(\Omega)$. We first recall a few notations and results from the
earlier works~\cite{FS,GS}.

\subsection{The linear semigroup and the axisymmetric
Biot-Savart law}\label{sec21}

As in \cite{GS}, we denote by $(S(t))_{t \ge 0}$ the evolution
semigroup defined by the linearized equation~\eqref{omeq} with
unit viscosity:
\begin{equation}\label{omlin}
  \partial_t \omega_\theta \,=\, \Bigl(\partial_r^2 + \partial_z^2
  + \frac{1}{r}\partial_r - \frac{1}{r^2}\Bigr)\omega_\theta\,,
\end{equation}
which is considered in the half-plane $\Omega = \{(r,z) \in \R^2 \,|\,
r > 0\,,~z \in\R\}$ with homogeneous Dirichlet boundary condition on
$\partial\Omega$. Using the explicit representation formula given in
\cite[Section~3]{GS}, one can show that the semigroup $(S(t))_{t \ge 0}$
is strongly continuous in $L^p(\Omega)$ for all $p \in [1,\infty)$,
and satisfies the same $L^p-L^q$ estimates as the heat semigroup in
$\R^2$. In particular, if $\omega_0 \in L^p(\Omega)$ for some $p \in
[1,\infty]$, then $S(t)\omega_0 \in L^q(\Omega)$ for all $t > 0$ and
all $q \in [p,\infty]$, and there exists a constant $C_2 > 0$ such that
\begin{equation}\label{Sest1}
  \|S(t)\omega_0\|_{L^q(\Omega)} \,\le\, \frac{C_2}{t^{\frac1p-\frac1q}}
  \,\|\omega_0\|_{L^p(\Omega)}\,, \qquad t > 0\,.
\end{equation}
Similarly, if $w = (w_r,w_z) \in L^p(\Omega)^2$, we have
\begin{equation}\label{Sest2}
  \|S(t)\div_* w\|_{L^q(\Omega)} \,\le\, \frac{C_2}{t^{\frac1p-\frac1q+\frac12}}
  \,\|w_0\|_{L^p(\Omega)}\,, \qquad t > 0\,,
\end{equation}
where $\div_* w = \partial_r w_r + \partial_z w_z$ denotes the
two-dimensional divergence of the vector field $w$.

On the other hand, if $\omega_\theta \in L^1(\Omega) \cap
L^\infty(\Omega)$, it is shown in \cite{FS,GS} that the
linear elliptic system~\eqref{diffBS}, with homogeneous
Dirichlet boundary condition for $u_r$ and homogeneous Neumann
condition for $u_z$, has a unique solution $u = (u_r,u_z) \in
C^0(\Omega)^2$ vanishing at infinity. Moreover $u \in L^q(\Omega)$
for all $q > 2$, and there exists a constant $C_3 > 0$ such that
\begin{equation}\label{BSest1}
  \|u\|_{L^\infty(\Omega)} \,\le\, C_3 \|\omega_\theta\|_{L^1(\Omega)}^{1/2}
  \,\|\omega_\theta\|_{L^\infty(\Omega)}^{1/2}\,.
\end{equation}
We call the map $\omega_\theta \mapsto u$ the axisymmetric Biot-Savart
law, and we occasionally denote $u = \BS[\omega_\theta]$. Explicit
formulas for $u$ in terms of $\omega_\theta$ can be found in
Section~2 of both references \cite{FS,GS}. We also recall the
following useful estimate: if $\omega_\theta \in L^1(\Omega)$ and
$\omega_\theta/r \in L^\infty(\Omega)$, then $u_r/r \in L^\infty(\Omega)$
and
\begin{align}\label{BSest2}
  \Bigl\|\frac{u_r}{r}\Bigr\|_{L^\infty(\Omega)} \,\le\,
  C_3\|\omega_\theta\|_{L^1(\Omega)}^{1/3}
  \,\|\omega_\theta/r\|_{L^\infty(\Omega)}^{2/3}\,.
\end{align}
Needless to say, both inequalities~\eqref{BSest1},~\eqref{BSest2}
are scale invariant.

Finally, it is important to note that, due to the divergence-free
condition in~\eqref{diffBS}, the evolution equation~\eqref{omeq}
can be written in the equivalent ``conservation form''
\begin{equation}\label{omeq2}
  \partial_t \omega_\theta + \div_*(u\,\omega_\theta)
  \,=\, \nu\Bigl(\partial_r^2 \omega_\theta + \partial_z^2 \omega_\theta
  + \partial_r\frac{\omega_\theta}{r}\Bigr)\,,
\end{equation}
where again $\div_*(u\,\omega_\theta) = \partial_r (u_r\omega_\theta) +
\partial_z (u_z\omega_\theta)$. We can thus define mild solutions
of~\eqref{omeq} in the following way:

\begin{df}\label{def-mild}
Given $T > 0$ and $\nu > 0$, we say that $\omega_\theta \in
C^0((0,T),L^1(\Omega) \cap L^\infty(\Omega))$ is a {\em mild
solution of~\eqref{omeq} on $(0,T)$} if the integral equation
\begin{equation}\label{omint}
  \omega_\theta(t) \,=\, S(\nu(t-t_0))\omega_\theta(t_0) - \int_{t_0}^t
  S(\nu(t-s))\div_*(u(s)\omega_\theta(s))\dd s
\end{equation}
is satisfied whenever $0 < t_0 < t < T$. Here $u(s) =
\BS[\omega_\theta(s)]$ for all $s \in (0,T)$.
\end{df}

In view of estimates~\eqref{Sest1},~\eqref{Sest2},~\eqref{BSest1},
it is clear that the right-hand side of~\eqref{omint} is
well-defined and belongs to $L^1(\Omega) \cap L^\infty(\Omega)$
for all admissible values of $t_0$ and $t$.

\subsection{A priori estimates}\label{sec22}

From now on, we always assume that $\omega_\theta \in C^0((0,T),
L^1(\Omega) \cap L^\infty(\Omega))$ is a mild solution of~\eqref{omeq}
on $(0,T)$ in the sense of Definition~\ref{def-mild}. We know from
\cite[Lemma~5.1]{GS} that the norm $\|\omega_\theta(t)\|_{L^1(\Omega)}$ is a 
nonincreasing function of time, and is even strictly decreasing unless
$\omega_\theta$ vanishes identically. We make the crucial assumption
that $\omega_\theta$ is {\em uniformly bounded in} $L^1(\Omega)$, so
that we can define
\begin{equation}\label{Mdef}
  M \,=\, \frac{1}{\nu}\lim_{t \to 0} \|\omega_\theta(t)\|_{L^1(\Omega)}
  \,<\, \infty\,.
\end{equation}
We thus have $\|\omega_\theta(t)\|_{L^1(\Omega)} \le M\nu$ for all $t \in (0,T)$. 
Under this hypothesis, the following a priori estimates hold. 

\begin{lem}\label{apriori-lem}
For any mild solution of~\eqref{omeq} on $(0,T)$ satisfying
\eqref{Mdef}, we have for all $t \in (0,T)$:
\begin{equation}\label{om-apriori}
  \Bigl\|\frac{\omega_\theta(t)}{r}\Bigr\|_{L^\infty(\Omega)} \,\le\,
  \frac{C_4 M}{t\sqrt{\nu t}}\,, \qquad \hbox{and}\quad
  \|\omega_\theta(t)\|_{L^\infty(\Omega)} \,\le\, \frac{C_5(M)M}{t}\,,
\end{equation}
where $C_4 > 0$ and $C_5 : [0,\infty) \to (0,\infty)$ is
increasing and polynomially bounded.
\end{lem}

\begin{proof}
It is sufficient to prove~\eqref{om-apriori} when $\nu = 1$,
because the general case then follows by a simple rescaling
argument. Due to parabolic smoothing, if $\omega_\theta \in
C^0((0,T),L^1(\Omega) \cap L^\infty(\Omega))$ is a mild solution
of~\eqref{omeq}, then $\omega_\theta$ is smooth on $\Omega \times
(0,T)$ and satisfies~\eqref{omeq} in the classical sense.
Applying Nash's method to the evolution equation satisfied
by the quantity $\omega_\theta/r$, one obtains the
following estimate:
\begin{equation}\label{apriori1}
  \Bigl\|\frac{\omega_\theta(t)}{r}\Bigr\|_{L^\infty(\Omega)} \,\le\,
  \frac{C_4}{(t-t_0)^{3/2}}\,\|\omega_\theta(t_0)\|_{L^1(\Omega)}
  \,\le\, \frac{C_4 M}{(t-t_0)^{3/2}}\,,
\end{equation}
for all $t \in (0,T)$ and all $t_0 \in (0,t)$, see
\cite[Lemma~3.8]{FS}. Thus taking the limit $t_0 \to 0$
we arrive at the first inequality in~\eqref{om-apriori}.
Similarly, it follows from \cite[Proposition~5.3]{GS} that
\begin{equation}\label{apriori2}
  \|\omega_\theta(t)\|_{L^\infty(\Omega)} \,\le\,
  \frac{C_5(\|\omega_\theta(t_0)\|_{L^1(\Omega)})}{t-t_0}\,
  \|\omega_\theta(t_0)\|_{L^1(\Omega)} \,\le\, \frac{C_5(M)M}{t-t_0}\,,
\end{equation}
where the function $C_5 : [0,\infty) \to (0,\infty)$ is increasing and
polynomially bounded. Taking again the limit $t_0 \to 0$ yields the
second inequality in~\eqref{om-apriori}.
\end{proof}

Combining Lemma~\ref{apriori-lem} with estimates~\eqref{BSest1},
\eqref{BSest2}, we easily obtain:

\begin{cor}\label{apriori-cor}
Under the assumptions of Lemma~\ref{apriori-lem}, we have
for all $t \in (0,T)$:
\begin{equation}\label{u-apriori}
  \Bigl\|\frac{u_r(t)}{r}\Bigr\|_{L^\infty(\Omega)} \,\le\,
  \frac{C_6 M}{t}\,, \qquad \hbox{and}\quad
  \|u(t)\|_{L^\infty(\Omega)} \,\le\, C_7(M)M \sqrt{\frac{\nu}{t}}\,,
\end{equation}
where $C_6 = C_3 C_4^{2/3}$ and $C_7(M) = C_3 C_5(M)^{1/2}$.
\end{cor}

We also have scale-invariant estimates on the derivatives of the
vorticity or the velocity. For instance, Proposition~5.5 in \cite{GS}
asserts that
\begin{equation}\label{derom-apriori}
  \|\nabla \omega_\theta(t)\|_{L^\infty(\Omega)} \,\le\,
  \frac{C_8(M)M}{t\sqrt{\nu t}}\,, \qquad 0 < t < T\,.
\end{equation}
More generally, the velocity $u = u_r e_r + u_z e_z$ (considered as a function 
of $x \in \R^3$) satisfies, for all $k,\ell \in \N$,
\begin{equation}\label{deru-apriori}
  \|\partial_t^k \nabla_x^\ell u(t)\|_{L^\infty(\R^3)} \,\le\,
  \frac{C_{k\ell}(M)M}{t^k (\nu t)^{\ell/2}}\sqrt{\frac{\nu}{t}}\,,
  \qquad 0 < t < T\,.
\end{equation}

\subsection{The trace of the solution at initial time}\label{sec23}

Using the a priori estimates established in the previous section,
we now prove that any mild solution satisfying~\eqref{Mdef}
converges as $t \to 0$ to some finite measure $\mu \in \cM(\Omega)$.

\begin{prop}\label{trace-prop}
If $\omega_\theta \in C^0((0,T),L^1(\Omega) \cap L^\infty(\Omega))$ is
a mild solution of~\eqref{omeq} on $(0,T)$ satisfying~\eqref{Mdef},
there exists a unique measure $\mu \in \cM(\Omega)$ such that
$\omega_\theta(t)\dd r\dd z \weakto \mu$ as $t \to 0$.
\end{prop}

\begin{proof}
We assume without loss of generality that $\nu = 1$. We first show
that $\omega_\theta(t)$ has a limit as $t \to 0$ in $D'(\Omega)$,
the space of all distributions on $\Omega$. Let $\phi \in
C^2_c(\Omega)$ be a $C^2$ function with compact support in
$\Omega$. Using~\eqref{omeq2} we find
\[
  \frac{\D}{\D t}\int_\Omega \phi\,\omega_\theta\dd r\dd z \,=\,
  \int_\Omega \Bigl(u\cdot\nabla\phi + \partial_r^2\phi +
  \partial_z^2\phi -\frac1r \partial_r\phi\Bigr)\omega_\theta\dd r\dd z\,,
\]
for all $t \in (0,T)$. As $\|\omega_\theta(t)\|_{L^1(\Omega)} \le M$
we have
\[
  \Bigl|\int_\Omega \Bigl(\partial_r^2\phi + \partial_z^2\phi
  -\frac1r \partial_r\phi\Bigr)\omega_\theta\dd r\dd z\Bigr|
  \,\le\, C M \|\nabla^2 \phi\|_{L^\infty(\Omega)}\,,
\]
and using estimate~\eqref{u-apriori} we also obtain
\[
  \Bigl|\int_\Omega u\cdot\nabla\phi\dd r\dd z\Bigr| \,\le\,
  \frac{C_7(M)M}{t^{1/2}}\,\|\nabla \phi\|_{L^\infty(\Omega)}\,.
\]
This shows that the quantity $\int_\Omega \phi(r,z)\,\omega_\theta(r,z,t)
\dd r\dd z$ has a limit as $t \to 0$ for any $\phi \in C^2_c(\Omega)$,
hence $\omega_\theta(t)$ converges in $D'(\Omega)$ to some limit
which we denote by $\mu$.

On the other hand, since $\omega_\theta(t)$ is uniformly bounded in
$L^1(\Omega)$ by assumption, the Banach-Alaoglu theorem asserts that,
for any sequence $t_m \to 0$, there exists a subsequence $t_m' \to 0$
and a measure $\bar\mu \in \cM(\Omega)$ such that $\omega_\theta(t_m')
\dd r\dd z\weakto \bar\mu$ as $m \to \infty$. But weak-star
convergence in $\cM(\Omega)$ implies convergence in $D'(\Omega)$, so
we necessarily have $\bar\mu = \mu$, hence $\mu \in
\cM(\Omega)$. Moreover, this shows that the weak-star limit is
independent of the choice of the subsequence $t_m' \to 0$, so that in
fact $\omega_\theta(t_m)\dd r\dd z \weakto \mu$ as $m \to
\infty$. Since the sequence $t_m \to 0$ was arbitrary, this is the
desired result.
\end{proof}

\begin{cor}\label{mild-cor}
Under the assumptions of Proposition~\ref{trace-prop},
one has
\begin{equation}\label{omint2}
  \omega_\theta(t) \,=\, S(\nu t)\mu - \int_0^t
  S(\nu(t-s))\div_*(u(s)\omega_\theta(s))\dd s\,,
  \qquad 0 < t < T\,.
\end{equation}
\end{cor}

\begin{proof}
We again assume that $\nu = 1$. For any fixed $t \in (0,T)$, our goal
is to take the limit $t_0 \to 0$ in the integral representation
\eqref{omint}, where both sides are considered as elements of
$L^1(\Omega)$. The integral term is easily controlled using estimates
\eqref{Sest2},~\eqref{BSest1}, and \eqref{u-apriori}. We find
\begin{align*}
  \int_{t_0}^t \|S(t-s)\div_*(u(s)\omega_\theta(s))\|_{L^1(\Omega)}\dd s
  \,&\le\, \int_{t_0}^t \frac{C_2}{(t-s)^{1/2}}\, \|u(s)\|_{L^\infty(\Omega)}
  \|\omega_\theta(s)\|_{L^1(\Omega)} \dd s\\
  \,&\le\, \int_0^t \frac{C_2 M}{(t-s)^{1/2}}\, \frac{C_7(M)M}{s^{1/2}}\dd s 
  \,=\, \pi C_2 C_7(M) M^2 \,<\, \infty\,,
\end{align*}
hence the integral term in~\eqref{omint} has a limit in $L^1(\Omega)$
as $t_0 \to 0$. To treat the other term, we decompose
\[
  S(t-t_0)\omega_\theta(t_0) \,=\, \Bigl(S(t-t_0) - S(t)\Bigr)
  \omega_\theta(t_0) + S(t)\omega_\theta(t_0)\,.
\]
Proceeding as in \cite[Section~3]{GS}, it is straightforward to prove that
\[
  \|(S(t-t_0) - S(t))\omega_\theta(t_0)\|_{L^1(\Omega)}
  \,\le\, C\,\frac{t_0}{t}\|\omega_\theta(t_0)\|_{L^1(\Omega)}
  \,\xrightarrow[~t_0\to 0~]{}\, 0\,.
\]
Moreover, using Proposition~\ref{trace-prop} and the explicit
representation formula for the semigroup $S(t)$ given in
\cite[Section~3]{GS}, we immediately obtain
\[
  \Bigl(S(t)\omega_\theta(t_0)\Bigr)(r,z) \,\xrightarrow[t_0\to 0]{}\,
  (S(t)\mu)(r,z)\,, \qquad \hbox{for all } (r,z) \in \Omega\,,
\]
and since the left-hand side of \eqref{omint} does not depend on $t_0$ 
we deduce that convergence holds in $L^1(\Omega)$ too. So taking
the limit $t_0 \to 0$ in~\eqref{omint} we obtain~\eqref{omint2}.
\end{proof}

\begin{rem}\label{gen-rem}
In view of Proposition~\ref{trace-prop}, a natural question
is whether a mild solution of~\eqref{omeq} on $(0,T)$ satisfying
\eqref{Mdef} is uniquely determined by its ``trace at initial time'',
namely by the measure $\mu$. Using the results established in \cite{GS},
one can show that the answer is positive if the atomic part of $\mu$
is small enough compared to viscosity. In the present paper, we focus
on the particular case where $\mu$ is a single Dirac mass.
The general case is still open.
\end{rem}

\subsection{The adjoint equation}\label{sec24}

In this section we consider the formal adjoint equation to~\eqref{omeq}
or~\eqref{omeq2} with respect to the scalar product in
$L^2(\Omega,\dd r\dd z)$, namely
\begin{equation}\label{adjeq}
  \partial_t \phi + u\cdot\nabla\phi + \nu\Bigl(\Delta \phi
  - \frac2r\,\partial_r \phi\Bigr) \,=\, 0\,.
\end{equation}
Eq.~\eqref{adjeq} can be solved backwards in time, imposing
simultaneously Dirichlet and Neumann boundary conditions on $\partial
\Omega$. There are several ways to see this. For example we can start
from the three-dimensional vorticity equation
\begin{equation}\label{v1}
  \omega_t + [u,\omega]-\nu \Delta \omega \,=\, 0\,,
  \qquad x \in \R^3\,,
\end{equation}
where we use the Lie bracket notation $[u,\omega]= (u\cdot\nabla)
\omega-(\omega\cdot\nabla)u$, and consider it as a linear equation for
(general) vector fields $\omega$, with $u$ given.  If we take the
adjoint equation to~\eqref{v1} for (general) vector fields $\Phi$,
given by the requirement that $\int_{\R^3} \omega(x,t)\cdot\Phi(x,t)\dd x$
be constant in time, we obtain
\begin{equation}\label{v2}
  \Phi_t + L_u\Phi + \nu\Delta \Phi \,=\, 0\,, \qquad x \in \R^3\,,
\end{equation}
where $L_u\Phi$ is the Lie derivative of $\Phi$ along the vector field
$u$ when $\Phi$ is considered as a $1$-form. In coordinates we have
$(L_u\Phi)_i = u_j\partial_j \Phi_i + \Phi_j\partial_iu_j$, where we sum
over repeated indices.

Due to estimates~\eqref{deru-apriori} we see from the standard linear
parabolic theory that Eq.~\eqref{v2} can be solved backwards in time,
for any bounded divergence-free ``terminal data'' $\Phi_1$ at time
$t_1 \in (0,T)$, and the solution $\Phi$ will be smooth in the open set
$\R^3\times(0,t_1)$. When $u$ is axisymmetric with no swirl, then the
fields $\omega$ of the form $\omega=\omega_\theta(r,z,t) e_\theta$ are
preserved by~\eqref{v1} (considered as a linear equation for
$\omega$), and the same is true for~\eqref{v2} and fields
$\Phi=\Phi_\theta(r,z,t)\,e_\theta$. For $\omega$ and $\Phi$ of this
form we have
\[
  \int_{\R^3} \omega(x,t)\cdot\Phi(x,t)\,\dd x \,=\, \int_{\Omega} 
  \omega_\theta(r,z,t)\Phi_\theta(r,z,t)2\pi r \dd r\dd z\,,
\]
and therefore in~\eqref{adjeq} we should take
\begin{equation}\label{v4}
  \phi \,=\, 2\pi r\Phi_\theta\,.
\end{equation}
For the solutions we consider here, equation~\eqref{adjeq} is the same
as~\eqref{v2} after the change of variables~\eqref{v4}. Now, as $\Phi$
is smooth in $\R^3\times (0,t_1)$, we must have $\Phi_\theta = r g$
for some smooth function $g=g(r,z,t)$ that is bounded on $\bar\Omega$
for any $t \in (0,t_1)$, and we conclude that the natural boundary
condition for $\phi$ at $r=0$ is that both $\phi$ and $\partial_r\phi$
vanish.

Alternatively, it is easy to verify that Eq.~\eqref{adjeq} is
well-posed (backwards in time) under the Neumann boundary condition
$\partial_r\phi(0,z,t) = 0$, and that the boundary data $a(z,t) =
\phi(0,z,t)$ satisfy the equation
\[
  \partial_t a(z,t) + u_z(0,z,t)\partial_z a(z,t) + \nu
  \partial_z^2 a(z,t) \,=\, 0\,, \qquad z \in \R\,, \qquad
  t \in (0,T)\,.
\]
In particular, if $\phi$ vanishes on the boundary $\partial \Omega$ at
terminal time $t_1$, the same property holds for all $t \in (0,t_1)$,
and as demonstrated above this is the natural condition under which
\eqref{adjeq} can be considered as the adjoint equation to
\eqref{omeq} in $\Omega$.

From now on, given $0 < t_0 < t_1 < T$ and $\phi_1 \in C_0(\Omega)$, we
denote by $\phi(r,z,t)$, for $(r,z) \in \Omega$ and $t \in (0,t_1)$,
the unique solution of \eqref{adjeq} with ``terminal condition''
$\phi(\cdot,\cdot,t_1) = \phi_1$. The main result of this subsection
is:

\begin{prop}\label{adj-prop}
Assume $u$ is the velocity field associated with a mild solution 
$\omega_\theta$ of \eqref{omeq} satisfying \eqref{Mdef}. Given $t_1 \in 
(0,T)$ and $\phi_1 \in C_0(\Omega)$, the unique solution $\phi$ of the 
(linear) adjoint equation~\eqref{adjeq} with terminal condition 
$\phi(\cdot,\cdot,t_1) = \phi_1$ can be extended to a continuous function 
on $\bar \Omega \times [0,t_1]$ satisfying $\phi(0,0,0) = 0$. Moreover 
one has $\phi(\cdot,\cdot,t) \in C_0(\Omega)$ for all $t \in [0,t_1]$, 
and
\begin{equation}\label{unifconvphi}
  \sup_{(r,z) \in \Omega} |\phi(r,z,t) - \phi(r,z,0)|
  \,\xrightarrow[t \to 0]{}\, 0\,.
\end{equation}
\end{prop}

\begin{proof}
We can assume that $\nu=1$ without loss of generality. As we have seen
above, the standard parabolic theory applied to the form~\eqref{v2}
of~\eqref{adjeq}, together with estimates~\eqref{deru-apriori} for $u$,
give the smoothness of $\phi$ for $t\in(0,t_1)$. The only issue is the
possible deterioration of the estimates as $t \to 0$. We will use
optimal regularity theory for linear parabolic equations with rough
coefficients to overcome the difficulty.

Consider the linear equation
\begin{equation}\label{p1}
  h_t + b(x,t)\cdot\nabla h + \Delta h \,=\, 0\,,
\end{equation}
in $Q = B\times(0,1)$, where $B\subset\R^n$ is a unit ball and $b$ is
a drift term. Assume that
\begin{equation}\label{p2}
  |\partial_t^k\nabla_x^\ell \,b| \,\le\, C_{k,\ell} \,t^{-k-\frac{\ell}{2}-\frac12}
  \quad \hbox{in }Q\,, \quad \hbox{for }k,\ell = 0,1,2,\dots
\end{equation}
This is a critical case for the regularity theory: if we could
increase the exponent on the right-hand side by any positive number,
no matter how small, the classical linear theory would imply that any
bounded solution $h$ is uniformly H\"older continuous in
$Q_r=B_r\times (0,r^2)$ for any $r<1$ (with estimates depending on
$r$). On the other hand, without additional assumptions the
condition~\eqref{p2} by itself may not be enough to arrive at that
conclusion.

Luckily, the velocity field $u$ in~\eqref{adjeq} has additional
properties. First, it is divergence-free. Second, it is bounded in the
space $L^\infty_t \BMO^{-1}_x$. It turns out that these two properties
are sufficient to ensure the H\"older-continuity estimates we
need. This is one of the main results in~\cite{SSSZ}, see also~\cite{FV}.
Strictly speaking, the claim in~\cite[Theorem~1.1]{SSSZ} is the
parabolic Harnack inequality, but it is well-known that H\"older
continuity is one of the easy consequences of the Harnack inequality.
In fact, in the present situation, one can even prove that $u$ is 
bounded in the space $L^\infty_t (L^\infty_x)^{-1}$, namely that $u = \div
\Psi$ for some matrix-valued function $\Psi$ that is bounded
in space and time. More precisely, we have the following result, 
whose proof is postponed to Section~\ref{sec52}.

\begin{lem}\label{lemma-bmo-bound}
Let $\omega = \omega_\theta e_\theta$ with $\omega_\theta\in L^1(\Omega)$, 
and let $u$ be the velocity field obtained from $\omega$ via the 
three-dimensional Biot-Savart law. Then there exists $c > 0$ such that
\begin{equation}\label{bmo-bound}
  \|u\|_{(L^\infty)^{-1}(\R^3)} \,\le\, c\|\omega_\theta\|_{L^1(\Omega)}\,.
\end{equation}
\end{lem}

Estimate \eqref{bmo-bound} is more than we need if we use the sharp
results of \cite{SSSZ}, but it has its own interest and it shows that
the more classical results of Osada \cite{Os} also apply to our
situation.

In what follows we denote
\begin{equation}\label{v6}
  \cO \,=\, \bigl\{(x_1,x_2,x_3) \in \R^3 \,\big|\, x_1^2 + x_2^2 > 0
  \bigr\}\,,
\end{equation}
namely $\cO$ is obtained from $\R^3$ be removing the symmetry axis
$x_1 = x_2 = 0$. We observe that the vector field $\frac1r e_r$ is
divergence-free and smooth in $\cO$. Together with~\eqref{bmo-bound}
this implies that in any parabolic ball $B_{x,\rho}\times
(0,\rho^2)\subset \cO\times(0,t_1)$ with $\rho < (x_1^2+x_2^2)^{1/2}$ the
adjoint equation \eqref{adjeq} is of the form~\eqref{p1} with $b\in L^\infty_t
\BMO^{-1}(B_{x,\rho})$ and $\div b=0$, so that Theorem 1.1 in~\cite{SSSZ}
can be applied.%
\footnote{Theorem 1.1 in~\cite{SSSZ} is purely local, even
though in the introduction of~\cite{SSSZ} a global condition $b\in
L^\infty_t\BMO^{-1}_x$ is mentioned. In the proof one only needs the
local condition. Also, when we are interested in the solution only
in $B_{x,\rho}\times (0,\rho^2)$, we can change the field $\frac 1r
e_r$ outside $B_{x,\rho}$ to a smooth div-free vector field in
$\R^3$, so that the global condition will in fact be satisfied (even
though it is not needed).} This remark will be used freely in the
proof below. Here and what follows, we consider $\phi$ and $\phi_1$
as functions on $\R^3 \times (0,t_1)$ and $\R^3$, respectively.

From the above considerations we see that our solution $\phi$ satisfies
the maximum principle:
\begin{equation}\label{max-princ}
  |\phi(x,t)| \,\le\, \max_{y \in \R^3}|\phi_1(y)|\,, \qquad x\in \R^3\,,
  \quad t\in(0,t_1]\,,
\end{equation}
and can be extended to a continuous function on $(\cO\times[0,t_1]) 
\cup (\R^3\times(0,t_1])$. The main point now is to prove its continuity 
at any point $(x,0)$ with $x \in \R^3\setminus\cO$. For any 
sufficiently small $\rho>0$, we define
\[
  A(\rho) \,=\, \sup \bigl\{\phi(x,t)\,\big|\, x \in \cC_\rho\,,~
  0< t < \rho^2\}\,,
\]
where $\cC_{\rho} = \bigl\{x\in\R^3\,\big|\, x_1^2+x_2^2\le\rho^2\}$
is the cylinder of radius $\rho$ centered on the $x_3$--axis. Clearly
$A$ is an increasing function, so that we can consider the
limit
\[
 a \,=\, \lim_{\rho\to 0} A(\rho) \,\ge\, 0\,.
\]
We need to show that $a=0$. (Once we have this, we can repeat the same
argument for $-\phi$, and we conclude that $\phi$ can be extended to a
continuous function on $\R^3 \times [0,t_1]$ satisfying $\phi(0,0) =
0$.) As in many other critical problems, it is natural to argue by
contradiction using the scaling invariance, see for example \cite{KNSS}
for a situation where related issues arise in the context of the
Navier-Stokes equations. 

Assume thus that $a>0$ and choose a sequence of points $(x^{(m)}, t^{(m)})$
such that $x^{(m)}$ approaches the $x_3$--axis, $t^{(m)}\searrow 0$, and
\[
  \lim_{m\to\infty}\phi(x^{(m)},t^{(m)}) \,=\, a\,.
\]
For $m \in \N$ we denote
\[
 \lambda_m \,=\, \sqrt{\bigl(x_1^{(m)}\bigr)^2+\bigl(x_2^{(m)}\bigr)^2+t^{(m)}}\,,
\]
and we define, for $y \in \R^3$ and $0 < s < \lambda_m^{-2}t_1$,
\begin{align}\label{ukdef}
 u^{(m)}(y,s) \,&=\,\lambda_m u(\lambda_m y_1,\lambda_m y_2,\lambda_m y_3+x_3^{(m)},
 \lambda_m^2 s)\,, \\ \label{phikdef}
 \phi^{(m)}(y,s) \,&=\,\,\,\, \,\,\,\phi(\lambda_m y_1,\lambda_m y_2,\lambda_m
  y_3+x_3^{(m)},\lambda_m^2 s)\,.
\end{align}
Setting
\[
  y^{(m)} \,=\, \bigl(\lambda_m^{-1}x_1^{(m)},\lambda_m^{-1}x_2^{(m)},0\bigr)\,,
  \qquad s^{(m)} \,=\, \lambda_m^{-2}t^{(m)}\,,
\]
we have $|y^{(m)}|^2+s^{(m)} = 1$ for all $m$, and we can therefore assume (after
extracting a subsequence, if necessary) that
\[
  (y^{(m)},s^{(m)}) \,\xrightarrow[m \to \infty]{}\,(\bar y,\bar s)\,, \qquad
  \hbox{where}\quad |\bar y|^2+\bar s \,=\, 1\,.
\]

Note that the operator $D := \frac2r \partial_r$ in \eqref{adjeq} has the same
scaling as the Laplacian, and is also invariant under translations along the
$x_3$--axis. Using this observation, it is straightforward to verify that
the functions $u^{(m)}, \phi^{(m)}$ defined in \eqref{ukdef}, \eqref{phikdef}
satisfy the equation
\begin{equation}\label{adjeqk}
  \partial_s \phi^{(m)} + u^{(m)}\cdot\nabla_y\phi^{(m)} + (\Delta_y -D_y)\phi^{(m)}
  \,=\, 0\,,
\end{equation}
in the region $\R^3 \times (0,\lambda_m^{-2}t_1)$. Moreover, in view of
\eqref{deru-apriori} and~\eqref{ukdef}, we have the a priori estimates
\begin{equation}\label{ukbounds}
  \|\partial_s^k \nabla_y^\ell \,u^{(m)}(s)\|_{L^\infty(\R^3)} \,\le\,
  C_{k\ell}(M)M \,s^{-k -\ell/2 -1/2}\,, \qquad
  0 < s < \lambda_m^{-2}t_1\,,
\end{equation}
which are similar to \eqref{p2} and hold uniformly in $m$. Finally,
uniform bounds on $\phi^{(m)}$ and its derivatives are easily obtained
by applying standard linear parabolic theory to Eq.~\eqref{adjeqk},
taking into account \eqref{max-princ} and \eqref{ukbounds}. Thus we
can assume (again after choosing subsequences, if necessary) that
\begin{equation}\label{v15}
  u^{(m)} \,\xrightarrow[m \to \infty]{}\, \bar u\,\,,\qquad
  \phi^{(m)} \,\xrightarrow[m \to \infty]{}\, \bar \phi\,,
\end{equation}
for suitable functions $\bar u, \bar \phi$, where the convergence is
uniform, with all derivatives, on compact subsets of
$\R^3\times(0,\infty)$. (Note that $u^{(m)},\phi^{(m)}$ are
well defined on any such set once $m$ is sufficiently large.)  By
construction, the functions $\bar u, \bar \phi$ satisfy $\partial_s
\bar \phi + \bar u\cdot\nabla \bar\phi + (\Delta -D)\bar \phi = 0$ in
$\R^3\times(0,\infty)$.  Due to~\eqref{bmo-bound} and scale-invariance
of the relevant norms we also have the uniform bound
\[
  \|u^{(m)}\|_{L^\infty_t\BMO^{-1}_x} \,\le\, c\|\omega_\theta\|_{L^\infty(0,T;L^1(\Omega))}\,,
\]
which means, as we have seen above, that the functions $\phi^{(m)}$
are in fact uniformly continuous up to $t=0$ on any set $B\times
[0,t_2]$ as long as $\bar B\subset \cO$. Hence $\bar\phi$ is
continuous on $(\cO\times[0,t_1]) \cup (\R^3\times(0,t_1])$, and it is
clear from the definitions that $\bar\phi \le a$ in that domain.  
At the same time, we know that
\[
  \bar\phi(\bar y,\bar s)  \,=\, \lim_{m\to\infty}\phi^{(m)}(y^{(m)},s^{(m)})
  \,=\, \lim_{m\to\infty}\phi(x^{(m)},t^{(m)}) \,=\, a\,.
\]
Finally, we have $\bar\phi(y,s) = \lim_{m\to \infty}\phi^{(m)}(y,s) = 0$ when 
$y \in \R^3\setminus\cO$ and $s > 0$. 

Since we assumed that $a>0$, these observations immediately lead
to a contradiction with the strong maximum principle when $\bar s>0$.
It thus remains to deal with the case where $\bar s = 0$ and $|\bar y|=1$.
In that situation, the Harnack inequality from~\cite[Theorem~1.1]{SSSZ}
applied to the parabolic ball $Q = B_{\bar y,1/2}\times [0,1/4)$ shows that
$\bar\phi = a$ in a neighborhood of $(\bar y,0)$ in $Q$, and we again
obtain a contradiction with the strong maximum principle, as in the
case $\bar s>0$. This concludes the proof of the assertion that
$\phi(x,t)$ extends to a continuous function on $\R^3 \times
[0,t_1]$ satisfying $\phi(0,0) = 0$.

To conclude the proof of Proposition~\ref{adj-prop}, it remains to
verify that $\phi(x,t)$ vanishes as $|x| \to \infty$ uniformly for all
$t \in (0,t_1]$, which implies in particular \eqref{unifconvphi} in
view of the previous results. Since $\phi_1 \in C_0(\Omega)$, this
property is intuitively obvious because the drift term in
Eq.~\eqref{adjeq} satisfies $\int_0^{t_1} \|u(\cdot,t)\|_{L^\infty}\dd
t < \infty$, and therefore can move ``diffusion particles'' over
finite distances only, during the time interval $(0,t_1)$.  This
heuristic argument can easily be made rigorous if one proceeds as in
\cite[Proposition~6.1]{GS}, see also Proposition~\ref{L1prop-far}
below.  Alternatively, it is possible to reach the same conclusion
using the parabolic Harnack inequality and the conservation of the
mass $\int_{\R^3} \phi(x,t)\dd x$, which can be checked by a direct
calculation. We leave the details to the reader.
\end{proof}

In the rest of this section, we derive a few important consequences
of Proposition~\ref{adj-prop}. By construction, if $\phi$ is as in
the statement, we have
\[
  \int_\Omega \omega_\theta(r,z,t) \phi(r,z,t) \dd r \dd z \,=\,
  \int_\Omega \omega_\theta(r,z,t_0) \phi(r,z,t_0) \dd r \dd z\,,
  \qquad 0 < t_0 \le t \le t_1\,.
\]
To take the limit $t_0 \to 0$, we decompose the right-hand side
as
\[
  \int_\Omega \omega_\theta(r,z,t_0) \bigl(\phi(r,z,t_0) -
  \phi(r,z,0)\bigr)\dd r \dd z + \int_\Omega \omega_\theta(r,z,t_0)
  \phi(r,z,0)\dd r\dd z\,,
\]
and we observe that the first term tends to zero in view of
\eqref{Mdef}, \eqref{unifconvphi} while the second one converges to
$\int_\Omega \phi(\cdot,\cdot,0)\dd\mu$ by
Proposition~\ref{trace-prop}. We thus have
\begin{equation}\label{adj-rel}
  \int_\Omega \omega_\theta(r,z,t) \phi(r,z,t) \dd r \dd z \,=\,
  \int_\Omega \phi(\cdot,\cdot,0)\dd\mu\,, \qquad 0 < t \le t_1\,.
\end{equation}

\begin{cor}\label{adj-cor1}
If $\omega_\theta \in C^0((0,T),L^1(\Omega) \cap L^\infty(\Omega))$ is
a mild solution of~\eqref{omeq} on $(0,T)$ satisfying~\eqref{Mdef},
then $\|\omega_\theta(t)\|_{L^1(\Omega)} \le \|\mu\|_{\tv}$ for all
$t \in (0,T)$. In particular, one has $M = \|\mu\|_{\tv}/\nu$.
\end{cor}

\begin{proof}
Fix $t_1 \in (0,T)$, and take $\phi_1 \in C_0(\Omega)$ such that
$\|\phi_1\|_{L^\infty(\Omega)} \le 1$. Let $\phi : \Omega \times [0,t_1]
\to \R$ be the solution of the adjoint equation~\eqref{adjeq}
with terminal condition $\phi(\cdot,\cdot,t_1) = \phi_1$
given by Proposition~\ref{adj-prop}. By the parabolic
maximum principle, we know that $|\phi(r,z,t)| \le 1$
for all $(r,z) \in \Omega$ and all $t \in [0,t_1]$. It thus
follows from~\eqref{adj-rel} with $t = t_1$ that
\[
  \Bigl|\int_\Omega \omega_\theta(r,z,t_1) \phi_1(r,z)\dd r\dd z
  \Bigr| \,=\, \Bigl|\int_\Omega \phi(\cdot,\cdot,0)\dd\mu\Bigr|
  \,\le\, \|\mu\|_\tv\,,
\]
and taking the supremum over all $\phi_1 \in C_0(\Omega)$
satisfying the bound $\|\phi_1\|_{L^\infty(\Omega)} \le 1$ we 
conclude that $\|\omega_\theta(t_1)\|_{L^1(\Omega)} \le \|\mu\|_\tv$. Thus
$M\nu = \lim_{t \to 0} \|\omega_\theta(t_1)\|_{L^1(\Omega)} \le
\|\mu\|_\tv$, and the converse inequality directly follows
from Proposition~\ref{trace-prop}.
\end{proof}

\begin{cor}\label{adj-cor2}
If the measure $\mu$ given by Proposition~\ref{trace-prop}
is positive, then the solution $\omega_\theta$ of
\eqref{omeq} satisfies $\omega_\theta(r,z,t) \ge 0$ for all
$(r,z) \in \Omega$ and all $t \in (0,T)$.
\end{cor}

\begin{proof}
Assume on the contrary that $\omega_\theta(r_1,z_1,t_1) < 0$
for some $(r_1,z_1) \in \Omega$ and some $t_1 \in (0,T)$.
Take a nonnegative function $\phi_1 \in C_0(\Omega)$ such that
$\phi_1(r_1,z_1,t_1) = 1$ and $\phi_1$ is supported in a small
neighborhood of $(r_1,z_1)$ where $\omega(\cdot,\cdot,t_1)$ takes
negative values only. If $\phi$ denotes the solution of the adjoint
equation~\eqref{adjeq} with terminal condition $\phi(\cdot,\cdot,
t_1) = \phi_1$, we obtain a contradiction from Eq.~\eqref{adj-rel}
with $t = t_1$ because the left-hand side is strictly negative
by construction, whereas the right-hand side is nonnegative
since $\phi \ge 0$ and $\mu$ is a positive measure.
\end{proof}

\begin{rem}\label{tight-rem}
An important consequence of Corollary~\ref{adj-cor1} is that
the family of signed measures $(\omega_\theta(\cdot,\cdot,t)\dd r
\dd z)_{t \in (0,T)}$ is tight, see Section~\ref{sec51} for more
details. In particular, this means that the convergence
\[
  \int_\Omega \phi(r,z) \omega_\theta(r,z,t) \dd r \dd z
   \,\xrightarrow[t \to 0]{}\, \int_\Omega \phi \dd\mu
 \]
holds for any bounded and continuous function $\phi$
on $\Omega$, and not just for any $\phi \in C_0(\Omega)$.
\end{rem}


\section{Gaussian estimates}\label{sec3}

As in the previous section, we assume that $\omega_\theta \in
C^0((0,T),L^1(\Omega) \cap L^\infty(\Omega))$ is a mild solution of
the axiymmetric vorticity equation~\eqref{omeq} on the time interval
$(0,T)$ satisfying~\eqref{Mdef}, and we denote by $\mu \in
\cM(\Omega)$ the initial measure defined by
Proposition~\ref{trace-prop}. Our goal here is to give accurate
estimates on the axisymmetric vorticity $\omega_\theta$ and the
associated velocity field $u = (u_r,u_z)$ under the additional
hypotheses that $\mu$ is a {\em positive} measure whose support is
{\em bounded away from the symmetry axis} $r = 0$ and {\em localized
  in the radial direction}. Of course, the application we have in mind
is the case where $\mu$ is a Dirac mass located at some point $(\bar
r,\bar z) \in \Omega$, which is the situation considered in
Theorem~\ref{main}.

\subsection{$L^1$ estimates near the symmetry axis}\label{sec31}

The goal of this section is to show that the $L^1$norm of the
axisymmetric vorticity $\omega_\theta$ is extremely small near the
symmetry axis for short times, if the initial measure $\mu$ is
positive and supported away from the axis. The precise statement is:

\begin{prop}\label{L1prop-near}
Assume that $\mu \in \cM(\Omega)$ is a positive measure whose support
is contained in the set $[2\rho,\infty) \times \R \subset \Omega$
for some $\rho > 0$. Then the solution
$\omega_\theta$ of~\eqref{omeq} satisfies
\begin{equation}\label{L1est-near}
  0 \,\le\, \int_0^{\rho}\left\{\int_\R \omega_\theta(r,z,t) \dd z\right\}
  \dd r \,\le\, C_9(M) \|\mu\|_\tv \,e^{-\frac{\rho^2}{16\nu t}}\,,
  \qquad t \in (0,T)\,,
\end{equation}
for some positive constant $C_9$ depending only on $M = \|\mu\|_\tv/\nu$.
\end{prop}

\begin{proof}
Without loss of generality we suppose that $\nu = 1$. Since $\mu$
is a positive measure, Corollary~\ref{adj-cor2} asserts that the solution
of~\eqref{omeq} satisfies $\omega_\theta(r,z,t) \ge 0$ for all
$(r,z) \in \Omega$ and all $t \in (0,T)$. As in \cite[Section~6.1]{GS},
we define
\begin{equation}\label{fdef}
  f(R,t) \,=\, \int_R^\infty \left\{\int_\R \omega_\theta(r,z,t)
  \dd z \right\}\dd r\,, \qquad R > 0\,,\quad t \in (0,T)\,.
\end{equation}
Then $f(R,t)$ is a nonincreasing function of $R$ which converges
to $\|\omega_\theta(t)\|_{L^1(\Omega)}$ as $R \to 0$ and
to zero as $R \to \infty$. Moreover $f$ satisfies the evolution
equation
\begin{equation}\label{evolf}
  \partial_t f(R,t) \,=\, \partial_R^2 f(R,t) + \frac1R\,
  \partial_R f(R,t) + \int_\R u_r(R,z,t) \omega_\theta(R,z,t)\dd z\,,
\end{equation}
which follows easily from~\eqref{omeq2}. Our goal is to obtain
a lower bound on $f(\rho,t)$ under the assumption that the
initial measure $\mu$ is supported in the set $[2\rho,\infty) \times
\R$. In view of Remark~\ref{tight-rem}, this hypothesis already implies
that $f(R,t) \to M = \|\mu\|_\tv$ as $t \to 0$ for any $R < 2\rho$.

Using the bound $\|u_r(t)\|_{L^\infty(\Omega)} \le C_7(M)M
t^{-1/2}$, which comes from Corollary~\ref{apriori-cor}, and
observing that $\partial_R f(R,t) = -\int_\R \omega_\theta(r,z,t)
\dd z \le 0$, we deduce from~\eqref{evolf} that
\begin{equation}\label{evolf2}
  \partial_t f(R,t) \,\ge\, \partial_R^2 f(R,t) + \frac1R\,
  \partial_R f(R,t) + C_7(M)\frac{M}{\sqrt{t}}\,\partial_R f(R,t)\,.
\end{equation}
To eliminate the drift terms in~\eqref{evolf2}, we fix $t_1 \in (0,T)$
and we define $g(y,t) = f(y+a(t),t)$ for $y \ge 0$ and $t \in (0,t_1]$,
where
\begin{equation}\label{adef}
   a(t) \,=\, \rho + \frac{t_1{-}t}{\rho} + 2C_7(M)M
  \bigl(\sqrt{t_1}-\sqrt{t}\bigr)\,,
   \qquad t \in [0,t_1]\,.
\end{equation}
Note that $a(t) \ge \rho$ for $t \in [0,t_1]$ and $a(t_1) = \rho$.
Using~\eqref{evolf2} and~\eqref{adef}, it is easy to verify that
\[
   \partial_t g(y,t) \,\ge\, \partial_y^2 g(y,t)\,, \qquad
   y \ge 0\,, \quad t \in (0,t_1]\,,
\]
and we obviously have $\partial_y g(0,t) = \partial_R f(a(t),t)
\le 0$ for $t \in (0,t_1]$. This implies that the function $g$ is
a {\em supersolution} of the heat equation on the half-line
$[0,\infty)$ with Neumann boundary condition at the origin.
More precisely, given any $t_0 \in (0,t_1)$, the parabolic
maximum principle implies that $g(y,t) \ge h(y,t)$ for all
$y \ge 0$ and all $t \in [t_0,t_1]$, where $h$ is defined by
\begin{equation}\label{hdef}
   \left\{\begin{array}{l}
   \hspace{0.3mm}\partial_t h(y,t)  \,=\, \partial_y^2 h(y,t)\,, \\
   \partial_y h(0,t) \,=\, 0\,, \\
   \hspace{2mm}h(y,t_0) \,=\, g(y,t_0) \,\equiv\, f(y+a(t_0),t_0)\,,
  \end{array}\right.
  \quad \begin{array}{l}
   y \ge 0\,, ~t \ge t_0\,, \\
   t \ge t_0\,, \\
   y \ge 0\,.
   \end{array}
\end{equation}
Solutions of~\eqref{hdef} are easily computed by symmetrizing
the initial data and solving the heat equation on the whole real
line. In particular, this gives the desired lower bound on
the quantity $f(\rho,t_1) = g(0,t_1)$.

To be more explicit, we first assume that the observation
time $t_1$ is small enough so that
\begin{equation}\label{t1small}
  4 t_1 \,\le\, \rho^2\,, \qquad \hbox{and}
  \qquad 8C_7(M)M \sqrt{t_1} \,\le\, \rho\,.
\end{equation}
In view of~\eqref{adef} we then have $a(t_0) \le a(0) \le 3\rho/2$
for any $t_0 \in (0,t_1)$, and this in turn implies that
$h(y,t_0) = f(y+a(t_0),t_0) \ge f(y+3\rho/2,t_0)$ for all
$y \ge 0$. Using the representation formula
\[
  h(0,t_1) \,=\, \frac{1}{\sqrt{\pi(t_1{-}t_0)}} \int_0^\infty
  e^{-\frac{y^2}{4t_1}}\,h(y,t_0)\dd y \,,
\]
and recalling that $f(y+3\rho/2,t_0) \to M$ as $t_0 \to 0$
for all $y < \rho/2$, we deduce that
\begin{equation}\label{flowerbound}
  f(\rho,t_1) \,\ge\, h(0,t_1) \,\ge\, \frac{M}{\sqrt{\pi t_1}}
  \int_0^{\rho/2}  e^{-\frac{y^2}{4t_1}} \dd y \,\ge\,
  M\Bigl(1 - e^{-\frac{\rho^2}{16t_1}}\Bigr)\,.
\end{equation}
In the last inequality we used the elementary bound
\begin{equation}\label{erfbound}
 \mathrm{erfc}(x) \,=\, \frac{2}{\sqrt{\pi}}
 \int_x^\infty e^{-y^2}\dd y \,\le\, e^{-x^2}\,, \qquad x \ge 0\,.
\end{equation}
Since $\|\omega_\theta(t_1)\|_{L^1(\Omega)} \le M = \|\mu\|_\tv$, we
conclude that
\[
  \int_0^{\rho}\left\{\int_\R \omega_\theta(r,z,t_1) \dd z\right\}
  \dd r \,\le\, M - f(\rho,t_1) \,\le\, M\,e^{-\frac{\rho^2}{16t_1}}\,,
\]
which gives the desired bound~\eqref{L1est-near} with $t = t_1$ and
$\nu = 1$, provided~\eqref{t1small} holds. If condition
\eqref{t1small} is not satisfied, one can take $C_9 = C_9(M) \ge
e^{\rho^2/(16 t_1)}$, in which case estimate~\eqref{L1est-near}
is obvious.
\end{proof}

\begin{cor}\label{nearaxis-cor}
Under the assumptions of Proposition~\ref{L1prop-near} we have
\begin{equation}\label{u/r-near}
  \Bigl\|\frac{u_r(t)}{r}\Bigr\|_{L^\infty(\Omega)} \,\le\,
  \frac{C_{10}(M)M}{t}\Bigl(\frac{\nu t}{\rho^2}\Bigr)^{1/3}\,,
  \qquad t \in (0,T)\,,
\end{equation}
where $C_{10}$ is a positive constant depending only on $M = \|\mu\|_\tv/\nu$.
\end{cor}

\begin{proof}
Fix $t \in (0,T)$. We decompose $\omega_\theta(r,z,t) =
\omega_\theta^-(r,z,t) + \omega_\theta^+(r,z,t)$, where
\[
  \omega_\theta^-(r,z,t) \,=\, \omega_\theta(r,z,t)\,\1_{\{r \le \rho\}}\,,
  \qquad \omega_\theta^+(r,z,t) \,=\, \omega_\theta(r,z,t)
  \,\1_{\{r > \rho\}}\,.
\]
By linearity of the axisymmetric Biot-Savart law, there is a
corresponding decomposition for the velocity field $u(r,z,t) =
u^-(r,z,t) + u^+(r,z,t)$, where $u^\pm$ is the velocity associated
with $\omega_\theta^\pm$, respectively. Using estimate~\eqref{BSest2},
Proposition~\ref{L1prop-near}, and the first inequality in
\eqref{om-apriori}, we find
\begin{align*}
  \Bigl\|\frac{u_r^-(t)}{r}\Bigr\|_{L^\infty(\Omega)} \,&\le\, C_3
  \|\omega_\theta^-(t)\|_{L^1(\Omega)}^{1/3} \,\|\omega_\theta^-(t)/r\|_{L^\infty(
  \Omega)}^{2/3} \\ \,&\le\, C_3\,C_4^{2/3}\,C_9(M)^{1/3}\,\frac{M}{t}
  \,e^{-\frac{\rho^2}{48\nu t}} \,\le\, \frac{C(M)M}{t}
  \,\Bigl(\frac{\nu t}{\rho^2}\Bigr)^{1/3}\,.
\end{align*}
Similarly, using the second inequality in~\eqref{om-apriori}, we
obtain
\begin{align*}
  \Bigl\|\frac{u_r^+(t)}{r}\Bigr\|_{L^\infty(\Omega)} \,&\le\, C_3
  \|\omega_\theta^+(t)\|_{L^1(\Omega)}^{1/3} \,\rho^{-2/3}\,\|\omega_\theta^+(t)
  \|_{L^\infty(\Omega)}^{2/3} \,\le\, C_3\,C_5(M)^{2/3}\,\frac{M}{t}
  \,\Bigl(\frac{\nu t}{\rho^2}\Bigr)^{1/3}\,.
\end{align*}
Combining both estimates we arrive at~\eqref{u/r-near}.
\end{proof}

\subsection{$L^1$ estimates away from the axis}\label{sec32}

We next consider the opposite case where the support of the
initial measure $\mu$ is bounded in the radial direction.
The analogue of Proposition~\ref{L1prop-near} is:

\begin{prop}\label{L1prop-far}
Assume that $\mu \in \cM(\Omega)$ is a positive measure whose support
is contained in the set $(0,2\rho] \times \R \subset \Omega$ for some
$\rho > 0$. Then the solution $\omega_\theta$ of~\eqref{omeq} satisfies
\begin{equation}\label{L1est-far}
  0 \,\le\, \int_{3\rho}^\infty\left\{\int_\R \omega_\theta(r,z,t) \dd z\right\}
  \dd r \,\le\, C_{11}(M) \|\mu\|_\tv \,e^{-\frac{\rho^2}{16\nu t}}\,,
  \qquad t \in (0,T)\,,
\end{equation}
for some positive constant $C_{11}$ depending only on $M = \|\mu\|_\tv/\nu$.
\end{prop}

\begin{proof}
We proceed as in the proof of Proposition~\ref{L1prop-near}, assuming
again that $\nu =1$. We observe that the function $f(R,t)$ defined
in~\eqref{fdef} satisfies the differential inequality
\begin{equation}\label{evolf3}
  \partial_t f(R,t) \,\le\, \partial_R^2 f(R,t) - C_7(M)\frac{M}{\sqrt{t}}
  \,\partial_R f(R,t)\,, \qquad R > 0\,,
\end{equation}
which is obtained in the same way as the lower bound~\eqref{evolf2}.
Arguing as in \cite[Section~6.1]{GS}, we deduce from~\eqref{evolf3}
that, for any $t_0 \in (0,T)$,
\[
  f(R,t) \,\le\, g(R - 2C_7(M)M\sqrt{t},t)\,, \qquad
  R > 0\,, \quad t_0 \le t   < T\,,
\]
where $g(y,t)$ is the solution of the heat equation $\partial_t g =
\partial_y^2 g$ on the real line $\R$ with initial data satisfying
$g(y,t_0) = f(y,t_0)$ if $y \ge 0$ and $g(y,t_0) = f(0,t_0)$ if
$y < 0$. Taking the limit $t_0 \to 0$ in the representation formula
\[
  g(y,t) \,=\, \frac{1}{\sqrt{4\pi(t{-}t_0)}}\biggl(
  \int_{-\infty}^0 e^{-\frac{(y-r)^2}{4(t-t_0)}}  f(0,t_0) \dd r + \int_0^\infty
  e^{-\frac{(y-r)^2}{4(t-t_0)}} f(r,t_0)\dd r\biggr)\,,
\]
and using the fact that $f(r,t_0) \le M$ and $f(r,t_0) \to 0$ as
$t_0 \to 0$ if $r > 2\rho$, we deduce that
\begin{equation}\label{fupp}
  f(R,t) \,\le\,  \frac{M}{\sqrt{4\pi t}}\int_{-\infty}^{2\rho}
  e^{-(R - 2C_7(M)M\sqrt{t}-r)^2/(4t)}\dd r\,, \qquad R > 0\,,\quad t \in (0,T)\,.
\end{equation}
If $t > 0$ is small enough so that $2C_7(M)M\sqrt{t} \le \rho/2$,
it follows from~\eqref{fupp},~\eqref{erfbound} that
\[
  f(3\rho,t) \,\le\, \frac{M}{\sqrt{4\pi t}}\int_{-\infty}^{2\rho}
  e^{-(5\rho/2-r)^2/(4t)}\dd r \,\le\, M\,e^{-\rho^2/(16t)}\,,
\]
which is~\eqref{L1est-far}. If $2C_7(M)M\sqrt{t} > \rho/2$,
then~\eqref{L1est-far} follows from the trivial bound
$f(3\rho,t) \le M$, provided the constant $C_{11}$ is
chosen appropriately.
\end{proof}

\begin{cor}\label{faraxis-cor}
Under the assumptions of Proposition~\ref{L1prop-far},
the axisymmetric vorticity $\omega_\theta$ has a finite impulse
\begin{equation}\label{imp-def}
 \cI \,=\, \int_\Omega r^2 \omega_\theta(r,z,t)\dd r\dd z \,=\,
 \int_\Omega r^2 \dd\mu(r,z)\,, \qquad t \in (0,T)\,.
\end{equation}
In particular, the impulse $\cI$ is a conserved quantity.
\end{cor}

\begin{proof}
We assume that $\nu =1$. Let $\chi : [0,\infty) \to \R$ be
a smooth, nonincreasing function such that $\chi(r) = 1$
for $r \in [0,1]$ and $\chi(r) = 0$ for $r \ge 2$. Using
definition~\eqref{fdef} and integrating by parts we obtain
the identity
\begin{equation}\label{omf1}
  \int_\Omega r^2 \chi(r/R) \omega_\theta(r,z,t)\dd r\dd z
  \,=\, \int_0^\infty r\Bigl(2\chi(r/R) + (r/R) \chi'(r/R)
  \Bigr) f(r,t)\dd r\,,
\end{equation}
which holds for all $R > 0$ and all $t \in (0,T)$. For any fixed $t
\in (0,T)$, we know from~\eqref{fupp} that $f(R,t)$ decays rapidly to
zero at infinity, thus taking the limit $R \to \infty$ in~\eqref{omf1}
we obtain
\begin{equation}\label{omf2}
  \int_\Omega r^2 \omega_\theta(r,z,t)\dd r\dd z \,=\,
  2 \int_0^\infty r f(r,t)\dd r \,<\, \infty\,, \qquad
  t \in (0,T)\,.
\end{equation}
The left-hand side of~\eqref{omf2} is the total impulse $\cI$
of the axisymmetric vorticity $\omega_\theta$, which is known
to be conserved under the evolution defined by~\eqref{omeq},
see e.g. \cite[Lemma~6.4]{GS}.

On the other hand, for any fixed $R > 2\rho$, the left-hand side
of~\eqref{omf1} converges as $t \to 0$ to the quantity
\[
  \cI_0 \,=\, \int_\Omega r^2\chi(r/R)\dd\mu(r,z) \,\equiv\,
  \int_\Omega r^2 \dd\mu(r,z)\,.
\]
Convergence holds by Remark~\ref{tight-rem}, and the limit does not
depend on $R > 2\rho$ since the measure $\mu$ is supported in
$(0,2\rho] \times \R$. In fact $\cI_0 = \cI$, because the convergence
of~\eqref{omf1} to~\eqref{omf2} as $R \to \infty$ holds {\em uniformly
in time} if $t > 0$ is sufficiently small. Indeed, if $2C_7(M)M\sqrt{t}
\le \rho$, it follows from~\eqref{fupp},~\eqref{erfbound} that
$f(R,t) \le M\,e^{-(R-3\rho)^2/(4t)}$ for all $R \ge 3\rho$, which in
turns implies that the quantity $\int_R^\infty rf(r,t)\dd r$ converges
to zero uniformly in time as $R \to \infty$. This proves the
uniform convergence of the right-hand side of~\eqref{omf1} to that
of~\eqref{omf2} as $R \to \infty$.
\end{proof}

The next step is a general estimate for nonnegative solutions
of~\eqref{omeq2} with finite impulse.

\begin{prop}\label{L1decay}
Assume that $\omega_\theta \in C^0((0,T),L^1(\Omega) \cap L^\infty(\Omega))$
is a nonnegative solution of \eqref{omeq2} which is uniformly bounded
in $L^1(\Omega)$ and has finite impulse $\cI$. Then
\begin{equation}\label{L1decayest}
  \|\omega_\theta(t)\|_{L^1(\Omega)} \,\le\, \frac{C_{12}(M)\cI}{\nu t}\,,
  \qquad \hbox{for all } t \in (0,T)\,,
\end{equation}
where $C_{12}$ is a positive constant depending only on the quantity
$M$ defined in~\eqref{Mdef}.
\end{prop}

\begin{proof}
The proof is essentially contained in \cite[Section~6.2]{GS}, although
estimate~\eqref{L1decayest} is not explicitly stated there. For
completeness we provide here the missing details, assuming as usual
that $\nu =1$. We first observe that it is sufficient to establish
\eqref{L1decayest} for $t \ge T_* = \cI/M$, because for smaller times
we obviously have $\|\omega_\theta(t)\|_{L^1(\Omega)} \le M \le \cI/t$.
We start from the integral equation~\eqref{omint} with $t_0 = t/2$, namely
\begin{equation}\label{omint3}
  \omega_\theta(t) \,=\, S(t/2)\omega_\theta(t/2) - \int_{t/2}^t
  S(t-s)\div_*(u(s)\omega_\theta(s))\dd s\,, \qquad
  t \ge T_*\,.
\end{equation}
To bound the first term in the right-hand side, we use the linear
estimate
\[
  \|S(t)\omega_0\|_{L^1(\Omega)} \,\le\, \frac{C}{t}\int_\Omega
  r^2 \omega_0(r,z)\dd r\dd z\,, \qquad t > 0\,,
\]
which holds for all nonnegative $\omega_0 \in L^1(\Omega)$ with finite
impulse, and can be established by a direct calculation based on the
explicit formula for the linear semigroup $S(t)$ given in
\cite[Section~3]{GS}. We thus have $\|S(t/2)\omega_\theta(t/2)
\|_{L^1(\Omega)} \le C\cI/t$ for some $C > 0$. On the other hand,
applying the weighted inequality given in \cite[Proposition~3.5]{GS},
we find
\[
  \|S(t-s)\div_*(u(s)\omega_\theta(s))\|_{L^1(\Omega)} \,\le\,
  \frac{C}{(t-s)^{3/4}}\,\|u(s)\|_{L^\infty(\Omega)}
  \|r^{1/2}\omega_\theta(s)\|_{L^1(\Omega)}\,,
\]
for $s \in (0,t)$. If we now interpolate $\|r^{1/2}\omega_\theta\|_{L^1}
\le \|r^2 \omega_\theta\|_{L^1(\Omega)}^{1/4} \|\omega_\theta\|_{L^1(\Omega)}^{3/4}$
and use the estimate
\[
  \|u\|_{L^\infty(\Omega)} \,\le\, C \|r^2 \omega_\theta\|_{L^1(\Omega)}^{1/4}
  \|\omega_\theta\|_{L^1(\Omega)}^{1/4} \|\omega_\theta/r\|_{L^\infty(\Omega)}^{1/2}\,,
\]
which is established in \cite[Section~2]{FS}, we obtain
\[
  \|S(t-s)\div_*(u(s)\omega_\theta(s))\|_{L^1(\Omega)} \,\le\,
  \frac{C}{(t-s)^{3/4}}\,\|r^2\omega_\theta(s)\|_{L^1(\Omega)}^{1/2}
  \|\omega_\theta(s)\|_{L^1(\Omega)} \|\omega_\theta(s)/r\|_{L^\infty(\Omega)}^{1/2}\,.
\]
As $\|r^2\omega_\theta(s)\|_{L^1(\Omega)} = \cI$ and $\|\omega_\theta(s)/r
\|_{L^\infty(\Omega)} \le C_4 Ms^{-3/2}$ by Lemma~\ref{apriori-lem}, we
deduce from~\eqref{omint3}
\begin{equation}\label{omint-est}
  \|\omega_\theta(t)\|_{L^1(\Omega)} \,\le\, \frac{C\cI}{t} +
  C M^{1/2}\cI^{1/2} \int_{t/2}^t \frac{\|\omega_\theta(s)\|_{L^1(\Omega)}}{
  (t-s)^{3/4}s^{3/4}}\dd s\,, \qquad t \ge T_*\,.
\end{equation}
The end of the proof is a straightforward bootstrap argument.
First, since $\|\omega_\theta(t)\|_{L^1(\Omega)} \le M$, estimate
\eqref{omint-est} shows that $\|\omega_\theta(t)\|_{L^1(\Omega)}
\le C(M)\cI^{1/2}t^{-1/2}$ for $t \ge T_*$, hence also for all
$t > 0$. Inserting this bound into the right-hand side of
\eqref{omint-est}, we conclude that $\|\omega_\theta(t)\|_{L^1(\Omega)}
\le C(M)\cI/t$, which is the desired result.
\end{proof}

\begin{cor}\label{faraxis-cor2}
Under the assumptions of Proposition~\ref{L1prop-far} we have
\begin{equation}\label{u/r-far}
  \Bigl\|\frac{u_r(t)}{r}\Bigr\|_{L^\infty(\Omega)} \,\le\,
  \frac{C_{13}(M)M}{t}\Bigl(\frac{\rho^2}{\nu t}\Bigr)^{1/3}\,,
  \qquad t \in (0,T)\,,
\end{equation}
where $C_{13}$ is a positive constant depending only on $M = \|\mu\|_\tv/\nu$.
\end{cor}

\begin{proof} Since $\supp(\mu) \subset (0,2\rho] \times \R$,
Corollary~\ref{faraxis-cor} shows that $\cI \le 4\rho^2 \|\mu\|_\tv =
4\rho^2 M\nu$. Thus estimate~\eqref{u/r-far} immediately follows
from~\eqref{BSest2},~\eqref{L1decayest}, and the first inequality
in~\eqref{om-apriori}.
\end{proof}

\subsection{Gaussian estimates for the viscous vortex ring}\label{sec33}

Finally, we consider the particular case where the initial measure
$\mu$ is a vortex filament located at some point $(\bar r, \bar z) \in
\Omega$, namely $\mu = \Gamma \delta_{(\bar r, \bar z)}$ for some
$\Gamma > 0$. We of course have $\|\mu\|_\tv = \Gamma$, hence
$M = \Gamma/\nu$, and $\cI = \Gamma \bar r^2$. The goal of this section
is to prove the following Gaussian estimate on the axisymmetric vorticity:

\begin{prop}\label{Gauss-prop}
Assume that $\omega_\theta \in C^0((0,T),L^1(\Omega) \cap L^\infty(\Omega))$
is a mild solution of~\eqref{omeq} which is uniformly bounded in
$L^1(\Omega)$, and such that $\omega_\theta(\cdot,t)\dd r\dd z \weakto
\Gamma \delta_{(\bar r, \bar z)}$ as $t \to 0$ for some $\Gamma > 0$ and some
$(\bar r, \bar z) \in \Omega$. For any $\eta \in (0,1)$ we have the
pointwise estimate
\begin{equation}\label{Gauss-est}
  0 \,<\, \omega_\theta(r,z,t) \,\le\, K_\eta(M)\,\frac{\Gamma}{\nu t}\,
  \exp\Bigl(-\frac{1-\eta}{4\nu t}\bigl((r-\bar r)^2 +
  (z-\bar z)^2\bigr)\Bigr)\,,
\end{equation}
for all $t \in (0,T)$ and all $(r,z) \in \Omega$, where the constant
$K_\eta(M)$ depends only on $\eta$ and $M = \Gamma/\nu$.
\end{prop}

As a first step in the proof of Proposition~\ref{Gauss-prop}, we apply the
results of Sections~\ref{sec31} and \ref{sec32} with $\rho = \bar r/2$
and obtain the following integral estimate:

\begin{lem}\label{u/r-lem}
Under the assumptions of Proposition~\ref{Gauss-prop} we have
\begin{equation}\label{u/r-int}
  \int_0^T \|u_r(t)/r\|_{L^\infty(\Omega)}\dd t \,\le\, C_{14}(M) M\,,
\end{equation}
where $C_{14}$ is a positive constant depending only on $M = \Gamma/\nu$.
\end{lem}

\begin{proof}
Let $T_* = \bar r^2/\nu = 4\rho^2/\nu$. Using estimate~\eqref{u/r-near}
for $t \in (0,T_*)$ and, if necessary, estimate~\eqref{u/r-far} for
$t \in (T_*,T)$, we immediately obtain~\eqref{u/r-int}.
\end{proof}

To derive estimate~\eqref{Gauss-est} it is convenient to abandon the
cylindrical coordinates and to return for a moment to the vector
valued vorticity $\omega(x,t) = \omega_\theta(r,z,t) e_\theta$,
which is considered as a function of $x = (r\cos\theta,r\sin\theta,z)
\in \R^3$ and $t \in (0,T)$. The evolution equation~\eqref{omeq}
is equivalent to
\begin{equation}\label{om-evol}
  \partial_t \omega + (U\cdot\nabla)\omega - V\omega \,=\,
  \nu \Delta \omega\,, \qquad x \in \R^3\,, \qquad t \in (0,T)\,,
\end{equation}
where $U = u_r e_r + u_z e_z$ is the velocity field associated with
$\omega$ via the three-dimensional Biot-Savart law, and $V = u_r/r$.
Since the pioneering work of Aronson \cite{Ar}, which relied itself on
previous results by Nash, De Giorgi, and Moser, it is well known that
solutions of advection-diffusion equations such as~\eqref{om-evol} can
be represented in terms of a (uniquely defined) fundamental solution
$\Phi$, which is H\"older continuous in space and time and satisfies
Gaussian upper and lower bounds.  In our problem we only have limited
information on the advection field $U$ and the potential $V$, and we
need an upper bound on the fundamental solution with explicit
dependence on the data $U$, $V$, and $\nu$. For that reason, we state
here a particular case of Aronson's estimates which is tailored to our
purposes.

\begin{prop}\label{fundsol-prop}
Assume that $U : \R^n \times (0,T) \to \R^n$ and $V : \R^n \times (0,T)
\to \R^n$ are continuous functions such that $\div U(\cdot,t) = 0$
for all $t \in (0,T)$ and
\begin{equation}\label{UVbounds}
  \sup_{0 < t < T} \Bigl(\frac{t}{\nu}\Bigr)^{1/2}
  \|U(\cdot,t)\|_{L^\infty(\R^n)} \,=\, K_1 \,<\, \infty\,, \qquad
  \int_0^T \|V(\cdot,t)\|_{L^\infty(\R^n)}\dd t \,=\, K_2 \,<\, \infty\,.
\end{equation}
Then the (regular) solutions of the advection-diffusion equation
\begin{equation}\label{f-evol}
  \partial_t f + (U\cdot\nabla)f - Vf \,=\,
  \nu \Delta f\,, \qquad x \in \R^n\,, \qquad t \in (0,T)\,,
\end{equation}
can be represented in the following way:
\[
  f(x,t) \,=\, \int_{\R^n} \Phi_{U,V,\nu}(x,t;y,s) f(y,s)\dd y\,,
  \qquad x \in \R^n\,, \qquad 0 < s < t < T\,,
\]
where the fundamental solution $\Phi_{U,V,\nu}(x,t;y,s)$ satisfies,
for $x,y \in \R^n$ and $0 < s < t < T$,
\begin{equation}\label{fund-bound}
  0 \,<\, \Phi_{U,V,\nu}(x,t;y,s) \,\le\, \frac{C_n}{(\nu(t{-}s))^{n/2}}
  \,\exp\Bigl(-\frac{|x-y|^2}{4\nu(t{-}s)} + K_1 \frac{|x-y|}{
  \sqrt{\nu(t{-}s)}} + K_2\Bigr)\,.
\end{equation}
Here the constant $C_n$ depends only on the space dimension $n$.
\end{prop}

For completeness, we give a short proof of Proposition~\ref{fundsol-prop}
in Section~\ref{sec53} below, but our purpose here is to apply it to the
vorticity equation~\eqref{om-evol}, for which $n = 3$.  In view of
Corollary~\ref{apriori-cor} and Lemma~\ref{u/r-lem}, both assumptions
in~\eqref{UVbounds} are satisfied, and the constants $K_1$, $K_2$
depend only on $M = \Gamma/\nu$. Solutions of~\eqref{om-evol} can thus
be represented in the following way:
\[
  \omega(x,t) \,=\, \int_{\R^3} \Phi(x,t;y,s) \omega(y,s)\dd y\,,
  \qquad x \in \R^3\,, \qquad 0 < s < t < T\,,
\]
and the fundamental solution $\Phi$ satisfies~\eqref{fund-bound}
with $n = 3$. As $\omega(x,t) = \omega_\theta(r,z,t) e_\theta$,
we deduce that the axisymmetric vorticity $\omega_\theta$
satisfies
\begin{equation}\label{omega-rep}
  \omega_\theta(r,z,t) \,=\, \int_\Omega \tilde\Phi(r,z,t;r',z',s)
  \omega_\theta(r',z',s )\dd r'\dd z'\,,
\end{equation}
for $(r,z) \in \Omega$ and $ 0 < s < t < T$, where
\begin{equation}\label{Phitilde-def}
  \tilde\Phi(r,z,t;r',z',s) \,=\, \int_{-\pi}^\pi
  \Phi([r,0,z],t;[r'\cos\theta,r'\sin\theta,z'],s)
  r' \cos\theta \dd\theta\,.
\end{equation}

\begin{lem}\label{Phitilde-lem}
For any $\eta \in (0,1)$ there exists a positive constant
$K_\eta(M)$, depending only on $\eta$ and $M$, such that
the fundamental solution $\tilde\Phi$ defined in~\eqref{Phitilde-def}
satisfies
\begin{equation}\label{Phitilde-est}
  0 \,<\, \tilde\Phi(r,z,t;r',z',s) \,\le\,
  \frac{K_\eta(M)}{\nu (t{-}s)}\,\frac{r'^{1/2}}{r^{1/2}}
  \,\tilde H\biggl(\frac{\nu (t{-}s)}{(1{-}\eta)rr'}\biggr)
  \,e^{-\frac{1-\eta}{4\nu (t-s)}\bigl((r-r')^2 +
  (z-z')^2\bigr)}\,,
\end{equation}
for $(r,z), (r',z') \in \Omega$ and $0 < s < t < T$,
where
\begin{equation}\label{Hdef}
  \tilde H(\tau) \,=\, \frac{1}{\sqrt{\pi \tau}} \int_{-\pi/4}^{\pi/4}
  e^{-\frac{\sin^2\phi}{\tau}} \cos(2\phi)\dd \phi\,,
  \qquad \tau > 0\,.
\end{equation}
\end{lem}

\begin{proof}
The positivity of the fundamental solution $\tilde\Phi$ of equation
\eqref{omeq} is a consequence of the strong maximum principle. To
obtain the upper bound~\eqref{Phitilde-est}, we first restrict the
integral in~\eqref{Phitilde-def} to the subdomain $[-\pi/2,\pi/2]$
where $\cos\theta \ge 0$, and then use estimate~\eqref{fund-bound}
with $n = 3$.  For any $\eta \in (0,1)$, it follows from
\eqref{fund-bound} and Young's inequality that
\[
  \Phi(x,t;y,s) \,\le\, \frac{C}{(\nu(t{-}s))^{3/2}}
  \,e^{-(1{-}\eta)\frac{|x-y|^2}{4\nu(t{-}s)} +
  \frac{K_1^2}{\eta} + K_2} \,=\, \frac{K_\eta(M)}{(\nu(t{-}s))^{3/2}}
  \,e^{-(1{-}\eta)\frac{|x-y|^2}{4\nu(t{-}s)}}\,.
\]
Taking $x = (r,0,z)$ and $y = (r'\cos\theta,r'\sin\theta,z')$,
we observe that
\[
  |x-y|^2 \,=\, |r-r'|^2 + |z-z'|^2 + 4rr' \sin^2(\theta/2)\,.
\]
Thus we deduce from~\eqref{Phitilde-def} that
\begin{align*}
  \tilde\Phi(r,z,t;r',z',s) \,&\le\,  \int_{-\pi/2}^{\pi/2}
  \Phi([r,0,z],t;[r'\cos\theta,r'\sin\theta,z'],s)
  r' \cos\theta \dd\theta \\
  \,&\le\, \frac{K_\eta(M)}{(\nu(t{-}s))^{3/2}}
   \,e^{-\frac{1-\eta}{4\nu (t-s)}\bigl((r-r')^2 +
  (z-z')^2\bigr)}
  \int_{-\pi/2}^{\pi/2} e^{-\frac{(1{-}\eta)rr'}{\nu(t{-}s)}
  \sin^2(\theta/2)}r' \cos\theta \dd\theta\,.
\end{align*}
Setting $\theta = 2\phi$ and using definition~\eqref{Hdef},
we arrive at~\eqref{Phitilde-est} with a modified constant
$K_\eta(M)$.
\end{proof}

\begin{rem}\label{Hrem}
The function $\tilde H$ in Lemma~\ref{Phitilde-lem} is not the same as
the function $H$ defined in \cite[Section~3]{GS}. One can show that
$\tilde H : (0,\infty) \to \R$ is decreasing with $\tilde H(\tau) \to
1$ as $\tau \to 0$ and $\tilde H(\tau) \sim 1/\sqrt{\pi \tau}$ as
$\tau \to \infty$. Moreover $\tilde H(\tau) \le 1/\sqrt{\pi \tau}$ for
all $\tau > 0$.
\end{rem}

\begin{proof}[\bf Proof of Proposition~\ref{Gauss-prop}.]
Fix $(r,z) \in \Omega$ and $t \in (0,T)$. Using the representation
\eqref{omega-rep} and the bound~\eqref{Phitilde-est}, we obtain
for all $s \in (0,t)$:
\[
  \omega_\theta(r,z,t) \,\le\, \frac{K_\eta(M)}{\nu (t{-}s)}
  \int_\Omega \frac{r'^{1/2}}{r^{1/2}}
  \,\tilde H\biggl(\frac{\nu (t{-}s)}{(1{-}\eta)rr'}\biggr)
  \,e^{-\frac{1-\eta}{4\nu (t-s)}\bigl((r-r')^2 +
  (z-z')^2\bigr)} \omega_\theta(r',z',s)\dd r'\dd z'\,.
\]
If $r' \le 2r$ in the right-hand side, we bound the function
$\tilde H$ by $1$. If $r' \ge 2r$ we use the fact that $\tilde H(\tau) \le
1/\sqrt{\pi \tau}$, so that
\[
  \frac{r'^{1/2}}{r^{1/2}}\,\tilde H\biggl(\frac{\nu(t{-}s)}{(1{-}\eta)r
  r'}\biggr) \,\le\, \frac{r'}{\sqrt{\pi}}\,
  \Bigl(\frac{1-\eta}{\nu (t{-}s)}\Bigl)^{1/2} \,\le\,
  C_\eta\,e^{\frac{\eta(1-\eta)}{4\nu (t-s)}(r-r')^2}\,,
\]
because $r'^2 \le 4(r-r')^2$. We thus obtain the simpler
estimate
\begin{align*}
  \omega_\theta(r,z,t) \,&\le\, \frac{K_\eta(M)}{\nu (t{-}s)}
  \int_\Omega e^{-\frac{(1-\eta)^2}{4\nu (t-s)}\bigl((r-r')^2 +
  (z-z')^2\bigr)} \omega_\theta(r',z',s)\dd r'\dd z' \\
  \,&\le\, \frac{K_\eta(M)}{\nu (t{-}s)} \int_\Omega
  e^{-\frac{(1-\eta)^2}{4\nu t}\bigl((r-r')^2 +
  (z-z')^2\bigr)} \omega_\theta(r',z',s)\dd r'\dd z'\,,
\end{align*}
with possibly a different constant $K_\eta(M)$. We now take the limit
$s \to 0$ and use the assumption that $\omega_\theta(\cdot,\cdot,s)
\dd r'\dd z' \weakto \Gamma \delta_{(\bar r, \bar z)}$, together with
Remark~\ref{tight-rem}. We thus obtain the upper bound in~\eqref{Gauss-est},
with a slightly different value of $\eta$. Finally, as was already 
observed, the positivity of $\omega_\theta$ is a consequence of the strong 
maximum principle. 
\end{proof}


\section{Self-similar variables and energy estimates}\label{sec4}

This section is is devoted to the actual proof of
Theorem~\ref{main}. We thus assume that $\omega_\theta \in C^0((0,T),
L^1(\Omega) \cap L^\infty(\Omega))$ is a mild solution of~\eqref{omeq}
which is uniformly bounded in $L^1(\Omega)$ and converges weakly to
$\Gamma \delta_{(\bar r, \bar z)}$ as $t \to 0$, for some $\Gamma > 0$
and some $(\bar r, \bar z) \in \Omega$. If $M$ is defined by
\eqref{Mdef}, we recall that $M = \Gamma/\nu$. The Gaussian estimate
in Proposition~\ref{Gauss-prop} indicates that, for short times, the
axisymmetric vorticity $\omega_\theta(r,z,t)$ concentrates in a
self-similar way around the initial position $(\bar r, \bar z)$ of the
vortex filament. A natural idea is thus to introduce self-similar
variables, in order to analyze more accurately the short-time behavior
of the solution.

\subsection{Definitions and a priori estimates}\label{sec41}

Motivated by~\eqref{Gauss-est}, we set
\begin{equation}\label{f-def}
  \omega_\theta(r,z,t) \,=\, \frac{\Gamma}{\nu t}\,f\Bigl(\frac{r-\bar r}
  {\sqrt{\nu t}}\,,\,\frac{z-\bar z}{\sqrt{\nu t}}\,,\,t\Bigr)\,, \qquad
  (r,z) \in \Omega\,,\quad t \in (0,T)\,.
\end{equation}
We also introduce the important notation
\begin{equation}\label{RZeps-def}
  \epsilon \,=\, \frac{\sqrt{\nu t}}{\bar r}\,, \qquad
  \gamma \,=\, \frac{\Gamma}{\nu}\,, \qquad
  R \,=\, \frac{r-\bar r}{\sqrt{\nu t}}\,, \qquad
  Z \,=\, \frac{z-\bar z}{\sqrt{\nu t}}\,.
\end{equation}
The dimensionless quantity $\epsilon$ is the ratio of the typical core
thickness of the vortex ring at time $t$ to the radius of the initial
vortex filament.  We are interested in the regime where $\epsilon$ is
small, and most of our analysis actually deals with the limit as
$\epsilon \to 0$. The ratio $\gamma$ of the vortex strength $\Gamma$
to the viscosity $\nu$ is sometimes called the ``circulation Reynolds
number'' in the physical literature. It is also dimensionless, and
coincides in the present case with the quantity $M$ defined in
\eqref{Mdef}, because we are considering positive solutions of
\eqref{omeq}. Finally, the dimensionless variables $(R,Z)$ are new
coordinates centered at the position of the vortex filament, where
distances are measured in units of the core thickness $\sqrt{\nu t}$.
Note that the domain constraint $r > 0$ translates into $1 + \epsilon
R > 0$, which means that the rescaled vorticity $f(R,Z,t)$ given by
\eqref{f-def} is actually defined in the time-dependent domain
$\Omega_\epsilon = \{(R,Z) \in \R^2\,|\, 1+\epsilon R > 0\}$, which
converges to $\R^2$ as $\epsilon \to 0$. However, since the function
$f(R,Z,t)$ satisfies the homogeneous Dirichlet condition at the
boundary $R = -1/\epsilon$, we can extended it by zero outside that
domain and thereby identify it with a function $\bar f(R,Z,t)$ which
is now defined on the whole plane $\R^2$, for any $t \in (0,T)$.

In view of~\eqref{Gauss-est}, the rescaled vorticity $f(R,Z,t)$
satisfies, for any $\eta \in (0,1)$,
\begin{equation}\label{Gaussf}
  0 \,<\, f(R,Z,t) \,\le\, K_\eta(M) \,e^{-\frac{1-\eta}{4}
  (R^2 + Z^2)}\,,
\end{equation}
for all $(R,Z) \in \Omega_\epsilon$ and all $t \in (0,T)$. Moreover,
it follows from~\eqref{derom-apriori} that the spatial
derivatives of $f$ are uniformly bounded:
\begin{equation}\label{unif-derf}
  |\nabla f(R,Z,t)| \,\le\, C_8(M)\,.
\end{equation}
Finally, using~\eqref{Mdef} and~\eqref{f-def}, we obtain
\begin{equation}\label{fmass}
  \int_{\Omega_\epsilon} f(R,Z,t) \dd R \dd Z \,=\,
  \frac{1}{\Gamma}\int_\Omega \omega_\theta(r,z,t)\dd r\dd z
  \,\xrightarrow[t \to 0]{}\, 1\,.
\end{equation}

It is also useful to express the velocity field $u$ associated
with $\omega_\theta$ in self-similar variables. The correct ansatz
is:
\begin{equation}\label{U-def}
  u(r,z,t) \,=\, \frac{\Gamma}{\sqrt{\nu t}}\,U^\epsilon\Bigl(\frac{r-\bar r}
  {\sqrt{\nu t}}\,,\,\frac{z-\bar z}{\sqrt{\nu t}}\,,\,t\Bigr)\,, \qquad
  (r,z) \in \Omega\,,\quad t \in (0,T)\,,
\end{equation}
where $U^\epsilon = U^\epsilon_r e_r + U^\epsilon_z e_z$ denotes the
rescaled velocity field. We use here the superscript $\epsilon$ to
keep in mind that, in the new variables, the Biot-Savart law depends
explicitly on time through the parameter $\epsilon = \sqrt{\nu t}/\bar r$.
Indeed, for any $t \in (0,T)$, the velocity $U^\epsilon$ satisfies the
elliptic system
\begin{equation}\label{U-BS}
  \partial_Z U^\epsilon_r - \partial_R U^\epsilon_z \,=\, f\,, \qquad
  \partial_R U^\epsilon_r + \frac{\epsilon U^\epsilon_r}{1+\epsilon R} +
  \partial_Z U^\epsilon_z \,=\, 0\,,
\end{equation}
in the domain $\Omega_\epsilon$, together with the boundary conditions
$U^\epsilon_r = \partial_R U^\epsilon_z = 0$ on $\partial\Omega_\epsilon$. In
view of~\eqref{u-apriori}, we have the following uniform a priori estimate
\begin{equation}\label{unif-Ueps}
  |U^\epsilon(R,Z,t)| \,\le\, C_7(M)\,, \qquad (R,Z) \in \Omega_\epsilon\,,
  \quad t \in (0,T)\,.
\end{equation}
In fact, estimate \eqref{unif-Ueps} can be improved as follows

\begin{lem}\label{decayU}
The rescaled velocity field defined in~\eqref{U-def} satisfies
\begin{equation}\label{decayU-est}
  (1 + |R| + |Z|)\,|U^\epsilon(R,Z,t)| \,\le\, C_{15}(M)\,, \qquad (R,Z)
  \in \Omega_\epsilon\,, \quad t \in (0,T)\,,
\end{equation}
where $C_{15}$ depends only on $M$.
\end{lem}

\begin{proof}
If $u$ is the velocity field associated with the vorticity
$\omega_\theta$ via the axisymmetric Biot-Savart law, it is shown in
\cite[Proposition~2.3]{GS} that
\[
  |u(r,z)| \,\le\, \int_\Omega \frac{C}{\sqrt{(r-r')^2 +
  (z-z')^2}}\,|\omega_\theta(r',z')|\dd r'\dd z'\,,
  \qquad (r,z) \in \Omega\,,
\]
where $C > 0$ is a universal constant. Using the change of variables
\eqref{f-def} and~\eqref{U-def}, we deduce that, for any $\epsilon > 0$,
\begin{equation}\label{Uepsbound}
  |U^\epsilon(R,Z)| \,\le\, \int_{\Omega_\epsilon} \frac{C}{\sqrt{(R-R')^2 +
  (Z-Z')^2}}\,|f(R',Z')|\dd R' \dd Z'\,, \qquad (R,Z) \in \Omega_\epsilon\,.
\end{equation}
In view of~\eqref{Uepsbound}, estimate~\eqref{decayU-est} follows
easily from the Gaussian bound~\eqref{Gaussf}.
\end{proof}

In Section~\ref{sec44} below we need accurate estimates on the
difference $U^\epsilon - U^0$, where $U^0$ denotes the velocity
field obtained from $f$ via the Biot-Savart law on $\R^2$. 
To prove such bounds, we need a rather explicit representation
for the solution of \eqref{U-BS}, which we now derive.

\subsection{The parametrized Biot-Savart law}\label{sec42}

We look for a solution of \eqref{U-BS} in the form
\begin{equation}\label{Ueps}
  U_r^\epsilon \,=\, -\frac{\partial_Z \phi^\epsilon}{1+\epsilon R}\,,
  \qquad
  U_z^\epsilon \,=\, \frac{\partial_R \phi^\epsilon}{1+\epsilon R}\,,
\end{equation}
where $\phi^\epsilon : \Omega_\epsilon \to \R$ is the axisymmetric
stream function, which satisfies the second-order elliptic
equation
\begin{equation}\label{phieps-eq}
  -\frac{\partial_R^2 \phi^\epsilon}{1+\epsilon R} +
  \frac{\epsilon \partial_R \phi^\epsilon}{(1+\epsilon R)^2}
  - \frac{\partial_Z^2 \phi^\epsilon}{1+\epsilon R} \,=\, f\,,
\end{equation}
in the domain $\Omega_\epsilon$, with both Dirichlet and Neumann
conditions on the boundary $\partial\Omega_\epsilon$. The solution
of \eqref{phieps-eq} can be computed as in \cite[Section~2]{FS}
and is found to be
\begin{equation}\label{phieps-rep}
  \phi^\epsilon(R,Z) \,=\, \frac{1}{2\pi}\int_{\Omega_\epsilon}
  \sqrt{(1{+}\epsilon R) (1{+}\epsilon R')} \,F\biggl(
  \epsilon^2\frac{(R{-}R')^2 + (Z{-}Z')^2}{(1{+}\epsilon R)
  (1{+}\epsilon R')}\biggr)f(R',Z')\dd R' \dd Z'\,,
\end{equation}
where $F : (0,\infty) \to (0,\infty)$ is defined by
\begin{equation}\label{Fexp}
  F(s) \,=\, \int_0^{\pi/2} \frac{1-2\sin^2\psi}{
  \sqrt{\sin^2\psi + s/4}}\dd\psi \,=\,
  \begin{cases} \log\frac{8}{\sqrt{s}} - 2 + \cO(s \log s)
  & \hbox{as }s \to 0\,,\\ \frac{\pi}{2s^{3/2}} + \cO(s^{-5/2})
  & \hbox{as }s \to \infty\,. \end{cases}
\end{equation}
Differentiating \eqref{phieps-rep} with respect to $R$ and $Z$,
and using \eqref{Ueps}, we obtain
\begin{align}\nonumber
  U^\epsilon_r(R,Z) \,&=\, \frac{1}{2\pi}\int_{\Omega_\epsilon}
  \sqrt{\frac{1{+}\epsilon R'}{1{+}\epsilon R}}
  \,\tilde F(\xi^2)\,\frac{Z-Z'}{(R{-}R')^2 + (Z{-}Z')^2}
  \,f(R',Z')\dd R' \dd Z'\,, \\ \label{Ueps-rep}
  U^\epsilon_z(R,Z) \,&=\, -\frac{1}{2\pi}\int_{\Omega_\epsilon}
  \sqrt{\frac{1{+}\epsilon R'}{1{+}\epsilon R}}
  \,\tilde F(\xi^2)\,\frac{R-R'}{(R{-}R')^2 + (Z{-}Z')^2}
  \,f(R',Z')\dd R' \dd Z' \\ \nonumber
  &\quad~ + \frac{\epsilon}{4\pi}\int_{\Omega_\epsilon}
  \frac{\sqrt{1{+}\epsilon R'}}{(1+\epsilon R)^{3/2}}
  \,\bigl(F(\xi^2) + \tilde F(\xi^2)\bigr) f(R',Z')\dd R' \dd Z'\,,
\end{align}
where $\xi^2$ is a shorthand notation for the quantity
\begin{equation}\label{xidef}
  \xi^2 \,=\, \epsilon^2\frac{(R{-}R')^2 + (Z{-}Z')^2}{(1{+}\epsilon R)
  (1{+}\epsilon R')}\,,
\end{equation}
and $\tilde F : (0,\infty) \to (0,\infty)$ is defined by
\begin{equation}\label{tildeFexp}
  \tilde F(s) \,=\, -2 s F'(s) \,=\, \begin{cases} 1 + \cO(s \log s)
  & \hbox{as }s \to 0\,,\\ \frac{3\pi}{2s^{3/2}} + \cO(s^{-5/2})
  & \hbox{as }s \to \infty\,. \end{cases}
\end{equation}
For simplicity we write $U^\epsilon = \BS^\epsilon[f]$ when \eqref{Ueps-rep}
holds.

When $\epsilon \to \infty$, the domain $\Omega_\epsilon$ shrinks to
the half-plane $\Omega$, and \eqref{Ueps-rep} coincides with the
axisymmetric Biot-Savart law, which is studied e.g. in
\cite[Section~2]{FS}. In contrast, as $\epsilon \to 0$, the domain
$\Omega_\epsilon$ expands to the full plane $\R^2$, and in this limit
\eqref{Ueps-rep} reduces to the usual two-dimensional Biot-Savart law:
\begin{equation}\label{BS2D}
  U^0(R,Z) \,=\, \begin{pmatrix} U^0_r(R,Z) \\
  U^0_z(R,Z)\end{pmatrix} \,=\, \frac{1}{2\pi} \int_{\R^2} \begin{pmatrix}
  Z-Z' \\ R'-R \end{pmatrix} \frac{f(R',Z')}{(R{-}R')^2 + (Z{-}Z')^2}
  \dd R'\dd Z'\,,
\end{equation}
which we denote $U^0 = \BS^0[f]$. Thus the $\epsilon$-dependent
Biot-Savart law defined by \eqref{U-BS} or \eqref{Ueps-rep} nicely
interpolates between the axisymmetric case and the two-dimensional
case.

We now compare the velocity fields $U^\epsilon$ and $U^0$ obtained
from the same vorticity distribution.

\begin{lem}\label{Udifflem}
Assume that $f$ vanishes outside $\Omega_\epsilon$. If $U^\epsilon =
\BS^\epsilon[f]$ and $U^0 = \BS^0[f]$, we have, for all $(R,Z) \in
\Omega_\epsilon$,
\begin{equation}\label{Udiffest}
  |U^\epsilon(R,Z) - U^0(R,Z)| \,\le\, \int_{\Omega_\epsilon} \frac{C\epsilon}{1
  + \epsilon R}\Bigl(1 + \logp\frac{1+\epsilon R}{\epsilon\rho}
  \Bigr) |f(R',Z')|\dd R' \dd Z'\,,
\end{equation}
where $\rho = \sqrt{(R{-}R')^2 + (Z{-}Z')^2}$ and $\logp(x) =
\max(\log(x),0)$.
\end{lem}

\begin{proof}
Since $f$ is supported in $\Omega_\epsilon$ by assumption, the
integrals in \eqref{Ueps-rep}, \eqref{BS2D} are taken over
the same domain. Thus, all we need is to subtract \eqref{BS2D} from
\eqref{Ueps-rep} and to estimate the various terms in the difference,
using the following elementary bounds
\begin{align}\label{Udiff1}
  &\Bigl|\sqrt{\frac{1{+}\epsilon R'}{1{+}\epsilon R}}\,\tilde F(\xi^2) - 1
  \Bigr| \,\le\, \sqrt{\frac{1{+}\epsilon R'}{1{+}\epsilon R}}\Bigl|
  \tilde F(\xi^2) - 1\Bigr| + \Bigl|\sqrt{\frac{1{+}\epsilon R'}{1{+}
  \epsilon R}} - 1\Bigr| \,\le\, C \frac{\epsilon \rho}{1{+}\epsilon R}\,,
  \\[2mm] \label{Udiff2}
  &\frac{\epsilon \sqrt{1{+}\epsilon R'}}{(1{+}\epsilon R)^{3/2}}
  \Bigl|F(\xi^2) + \tilde F(\xi^2)\Bigr| \,\le\, C\frac{\epsilon}{1{+}
  \epsilon R} \Bigl(1 + \logp\frac{1{+}\epsilon R}{\epsilon \rho}\Bigr)\,.
\end{align}
Estimate \eqref{Udiff1} easily follows from the bound $|\tilde
F(\xi^2) - 1| \le C|\xi|$, which is a direct consequence of
\eqref{tildeFexp}. The proof of \eqref{Udiff2} requires a little more
work. In the region where $1+\epsilon R' \le 2(1+\epsilon R)$, we
obtain \eqref{Udiff2} using the facts that $\tilde F(\xi^2)$ is
bounded and $F(\xi^2) \le C(1+\logp \xi^{-1})$, see \eqref{Fexp}. When
$1+\epsilon R' \ge 2(1+\epsilon R)$, we observe that $2\epsilon\rho
\ge 2\epsilon(R'-R) \ge 1+\epsilon R'$, and using the bounds $F(\xi^2)
+ \tilde F(\xi^2) \le C \xi^{-1}$ we obtain \eqref{Udiff2} (without
the logarithmic term in that case).
\end{proof}

\subsection{Characterization of the $\alpha$-limit set}\label{sec43}

The evolution equation satisfied by the rescaled vorticity $f$ defined
in~\eqref{f-def} reads
\begin{equation}\label{feq}
  t \partial_t f + \gamma \partial_R(U^\epsilon_r f) + \gamma
  \partial_Z(U^\epsilon_z f) \,=\, \cL f + \partial_R
  \Bigl(\frac{\epsilon f}{1+\epsilon R}\Bigr)\,,
\end{equation}
for $(R,Z) \in \Omega_\epsilon$ and $t \in (0,T)$, where $\gamma =
\Gamma/\nu$ and $\cL$ is the differential operator defined by
\begin{equation}\label{cLdef}
  \cL f \,=\, (\partial_R^2 + \partial_Z^2)f + \frac12
  (R\partial_R f + Z\partial_Z f) + f\,.
\end{equation}
The homogeneous Dirichlet boundary condition for $f$ reads
$f(-1/\epsilon,Z,t) = 0$ for all $Z \in \R$ and all $t \in (0,T)$. If
we formally take the limit $\epsilon \to 0$ in~\eqref{feq},
\eqref{U-BS} and introduce the fictitious time $\tau = \log(t/T)$, so
that $\partial_\tau = t\partial_t$, we arrive at the evolution
equation
\begin{equation}\label{flimit}
  \partial_\tau f + \gamma U \cdot \nabla f \,=\, \cL f\,, \qquad (R,Z) \in
  \R^2\,,
\end{equation}
where $\partial_R U_r + \partial_Z U_z = 0$ and $\partial_Z U_r - \partial_R U_z
= f$. In other words, we obtain in that limit the two-dimensional vorticity
equation in self-similar variables, which was thoroughly studied, for
instance, in \cite{GW1,GW2}.

We now introduce the weighted $L^2$ space $X = \{f \in L^2(\R^2)\,|\,
\|f\|_X < \infty\}$ where
\begin{equation}\label{Xdef}
  \|f\|_X^2 \,=\, \int_{\R^2} |f(R,Z)|^2 \,e^{(R^2+Z^2)/4}\dd R \dd Z\,.
\end{equation}
For later use we also denote
\begin{equation}\label{wGdef}
  w(R,Z) \,=\, e^{(R^2+Z^2)/4}\,, \qquad G(R,Z) \,=\, \frac{1}{4\pi}
   e^{-(R^2+Z^2)/4}\,, \qquad (R,Z) \in \R^2\,.
\end{equation}
The aim of this section is to prove the following result:

\begin{prop}\label{alpha-prop}
The solution of~\eqref{feq} defined by~\eqref{f-def} satisfies
$\|\bar f(t) - G\|_X \to 0$ as $t \to 0$, where $\bar f$ denotes the extension
of $f$ by zero outside $\Omega_\epsilon$.
\end{prop}

Proposition~\ref{alpha-prop} means that the axisymmetric vorticity
$\omega_\theta(r,z,t)$ is not only bounded from above by a
self-similar function with Gaussian profile, as asserted in
\eqref{Gaussf}, but actually approaches a uniquely determined
self-similar solution of the 2d vorticity equation as $t \to
0$. Before giving a detailed proof, we make some preliminary
remarks. Let $X_0 \subset X$ be the Banach space defined by the norm
\begin{equation}\label{X0def}
  \|f\|_{X_0} \,=\, \sup_{(R,Z) \in \R^2} |f(R,Z)|\, e^{\frac{1-\eta}{4}(R^2+Z^2)}
  ~+ \sup_{(R,Z) \in \R^2} |\nabla f(R,Z)|\,,
\end{equation}
where $\eta \in (0,1/2)$ is any fixed real number. We have the
following elementary result:

\begin{lem}\label{elem}
The space $X_0$ is compactly embedded in $X$, and the unit ball in $X_0$
is closed for the topology induced by $X$.
\end{lem}

According to~\eqref{Gaussf} and~\eqref{unif-derf} the trajectory
$(\bar f(t))_{t \in (0,T)}$ is bounded in $X_0$, hence relatively
compact in $X$.  We can thus consider the $\alpha$-limit set
\[
  \cA \,=\, \bigl\{h \in X\,\big|\, \hbox{there exists a sequence }
  t_m \to 0 \hbox{ such that } \|\bar f(t_m) - h\|_X \to 0
  \hbox{ as }m \to \infty\bigr\}\,,
\]
which is of course nonempty. We know from Lemma~\ref{elem} that $\cA$
is bounded in $X_0$, and in view of~\eqref{fmass} any $h \in \cA$
satisfies $\int h \dd R \dd Z = 1$. Proposition~\ref{alpha-prop}
asserts that $\cA$ is a {\em singleton}, namely $\cA = \{G\}$. The
intuition behind this result is that the $\alpha$-limit set $\cA$ is
(positively and negatively) invariant under the evolution defined on
the whole plane $\R^2$ by the limiting equation~\eqref{flimit}, which
is obtained by formally taking the limit $\epsilon \to 0$ in
\eqref{feq}. But it is proved in \cite{GW2} that the only solutions of
\eqref{flimit} that are uniformly bounded in $X$ for all negative times 
$\tau$ are equilibria of the form $f = \alpha G$, for some $\alpha \in \R$. 
Since we have the normalization condition $\int h \dd R \dd Z = 1$ for 
any $h \in \cA$, we conclude that $\cA = \{G\}$.

In fact, comparing the evolutions defined by equations~\eqref{feq} and
\eqref{flimit} requires some work, so we prefer using a different
argument to establish Proposition~\ref{alpha-prop}.

\begin{proof}[\bf Proof of Proposition~\ref{alpha-prop}.] Let $h_*
\in \cA$, and let $(t_m)$ be a sequence in $(0,T)$ such that $t_m \to
0$ and $\|\bar f(t_m) - h_*\|_X \to 0$ as $m \to \infty$. Our goal is to
show that $h_* = G$. To prove that, it is convenient to return to the
three-dimensional formulation of the vorticity equation. As in
\eqref{uomaxi}, we denote by $u(x,t)$ and $\omega(x,t)$ the
three-dimensional velocity and vorticity fields, respectively.
For any $m \in \N$, any $y \in \R^3$, and any $s \in (0,T\epsilon_m^{-2})$,
we define
\begin{equation}\label{omum_def}
 \left\{\begin{array}{l}
  \,u^{(m)}(y,s) \,=\, \epsilon_m \,u(\bar x + \epsilon_m y\,,\,\epsilon_m^2s)\,,
  \\[1mm] \omega^{(m)}(y,s) \,=\, \epsilon_m^2\,\omega(\bar x+\epsilon_my\,,\,
  \epsilon_m^2s)\,, \end{array}\right.
\end{equation}
where $\bar x = (\bar x_1, \bar x_2, \bar x_3) = (\bar r,0,\bar z) \in \R^3$
and, in agreement with~\eqref{RZeps-def},
\[
   \epsilon_m \,=\, \frac{\sqrt{\nu t_m}}{\bar r}\,, \qquad
   m \in \N\,.
\]
In other words, the vector fields $u^{(m)}, \omega^{(m)}$ are defined
by a self-similar blow-up of the original quantities $u,\omega$
near the point $\bar x \in \R^3$ and near the initial time $t = 0$.

It is clear that $u^{(m)}, \omega^{(m)}$ satisfy the three-dimensional
vorticity equation
 \begin{equation}\label{uom_eq}
   \partial_s \omega^{(m)} + [u^{(m)},\omega^{(m)}] -\nu\Delta \omega^{(m)}
  \,=\, 0\,,
\end{equation}
for $y \in \R^3$ and $0 < s < T\epsilon_m^{-2}$, together with the
constraints $\div u^{(m)} = 0$ and $\curl u^{(m)} = \omega^{(m)}$.
This is due to the scaling and translational symmetries of the
equations. Note that, in~\eqref{uom_eq} and in the rest of the proof,
all spatial derivatives act on the variable $y \in \R^3$. In view of
\eqref{deru-apriori} and~\eqref{omum_def}, we have the a priori
estimates
\[
  \|\partial_s^k \nabla_y^\ell u^{(m)}(s)\|_{L^\infty(\R^3)} \,\le\,
  \frac{C_{k\ell}(M)M}{s^k (\nu s)^{\ell/2}}\sqrt{\frac{\nu}{s}}\,,
  \qquad 0 < s < T\epsilon_m^{-2}\,,
\]
which hold for all indices $k,\ell \in \N$, uniformly in $m \in \N$. Up to
extracting a subsequence, we can therefore assume that
\[
   u^{(m)}\to \bar u\,, \qquad \omega^{(m)}\to\bar\omega\,,
   \qquad \hbox{as}\quad m \to \infty\,,
\]
with uniform convergence of both vector fields and all their derivatives
on any compact subset of $\R^3 \times (0,\infty)$. The limiting fields
$\bar u,\bar \omega$ are smooth on $\R^3 \times (0,\infty)$ and satisfy
\begin{equation}\label{uom_limit}
  \partial_s \bar\omega + [\bar u,\bar\omega] - \nu\Delta\bar \omega \,=\, 0\,,
\end{equation}
together with $\div\bar u=0$ and $\curl\bar u =\bar \omega$.

We now relate the limiting vorticity $\bar\omega$ to the $\alpha$-limit
points of the rescaled vorticity $f$. In view of~\eqref{f-def} and
\eqref{omum_def}, we have, for all $m \in \N$, all $y \in \R^3$, and
all $s \in (0,T\epsilon_m^{-2})$,
\begin{equation}\label{omm_rep}
  \omega^{(m)}(y,s) \,=\, \frac{\Gamma}{\nu s}f\biggl(
  \frac{\sqrt{(\bar r + \epsilon_m y_1)^2 + \epsilon_m^2 y_2^2}-\bar r}{
  \epsilon_m \sqrt{\nu s}}\,,\frac{y_3}{\sqrt{\nu s}}\,,\epsilon_m^2 s
  \biggr)\,e_\theta(\bar x + \epsilon_m y)\,.
\end{equation}
For any fixed $s > 0$, we can assume (up to extracting another
subsequence) that $f(\cdot,\cdot,\epsilon_m^2 s)$ converges in the
topology of $X$ to some $h_s \in \cA$ as $m \to \infty$. Since
$f(\cdot,\cdot,t)$ is bounded in $X_0$, the convergence also holds
uniformly on any compact set of $\R^3$. Thus taking the limit
$m \to \infty$ in~\eqref{omm_rep} and observing that $e_\theta(\bar x)
= e_2 = (0,1,0)$, we obtain
\begin{equation}\label{barom_rep}
  \bar \omega(y,s) \,=\, \frac{\Gamma}{\nu s} h_s\Bigl(\frac{y_1}{\sqrt{\nu s}}
  \,,\frac{y_3}{\sqrt{\nu s}}\Bigr)\,e_2 \,=:\, \bigl(0,\bar \omega_2(y_1,y_3,s),
  0\bigr)\,.
\end{equation}
We deduce in particular that
\begin{equation}\label{barom_prop}
  |\bar \omega_2(y_1,y_3,s)| \,\le\, K_\eta(M)\frac{\Gamma}{\nu s}
  \,e^{-\frac{1-\eta}{4\nu s}(y_1^2 + y_3^2)}\,, \quad \hbox{and}\quad
  \int_{\R^2} \bar \omega_2(y_1,y_3,s)\dd y_1\dd y_3 \,=\, \Gamma\,.
\end{equation}

Similarly, in view of~\eqref{U-def} and~\eqref{omum_def}, the velocity field
satisfies
\[
  u^{(m)}(y,s) \,=\, \frac{\Gamma}{\sqrt{\nu s}}\,U^{\epsilon_m}\biggl(
  \frac{\sqrt{(\bar r + \epsilon_m y_1)^2 + \epsilon_m^2 y_2^2}-\bar r}{
  \epsilon_m \sqrt{\nu s}}\,,\frac{y_3}{\sqrt{\nu s}}\,,\epsilon_m^2 s
  \biggr)\,.
\]
Applying estimate~\eqref{decayU-est} to the right-hand side and
taking the limit $m \to \infty$, we obtain
\begin{equation}\label{baru_est}
  |\bar u(y,s)| \,\le\, \frac{C_{15}(M)\Gamma }{\sqrt{\nu s} + |y_1| + |y_3|}\,,
  \qquad y \in \R^3\,, \quad s > 0\,.
\end{equation}
If follows from these observations that $\bar u$ is given by
\begin{equation}\label{baru_rep}
  \bar u(y,s) \,=\, \bar u_1(y_1,y_3,s) e_1 + \bar u_3(y_1,y_3,s) e_3 \,=\,
  \bigl(\bar u_1(y_1,y_3,s),0,\bar u_3(y_1,y_3,s)\bigr)\,,
\end{equation}
where $(\bar u_1,\bar u_3)$ is the two-dimensional velocity field
obtained from the scalar vorticity $\bar \omega_2$ via the Biot-Savart
law in $\R^2$. Indeed, the velocity field defined by~\eqref{baru_rep}
satisfies $\partial_1 \bar u_1 + \partial_3 \bar u_3 = 0$ and
$\partial_3 \bar u_1 - \partial_1 \bar u_3 = \bar \omega_2$. We thus
have $\div \bar u = 0$ and $\curl \bar u = \bar\omega$, and that
elliptic system has no other solution with the property
\eqref{baru_est}.

Summarizing, we have shown that the limiting vorticity $\bar\omega_2$,
together with the associated velocity $(\bar u_1,\bar u_3)$, solves the
Navier-Stokes equations in $\R^2 \times (0,\infty)$, and it follows from
\eqref{barom_prop} that $\bar\omega_2(\cdot,s)$ is uniformly bounded
in $L^1(\R^2)$ and converges weakly to the Dirac measure $\Gamma
\delta_0$ as $s \to 0$. Invoking \cite[Proposition~1.3]{GW2}, we
deduce that, for any $s > 0$,
\begin{equation}\label{barom_fin}
  \bar\omega_2(y_1,y_3,s) \,=\, \frac{\Gamma}{\nu s}\,G\Bigl(
  \frac{y_1}{\sqrt{\nu s}}
  \,,\frac{y_3}{\sqrt{\nu s}}\Bigr) \,=\, \frac{\Gamma}{4\pi\nu s}
  \,e^{-(y_1^2 + y_3^2)/(4\nu s)}\,, \qquad (y_1,y_3) \in \R^2\,.
\end{equation}
In particular, setting $s = s_* = \bar r^2/\nu$, so that $\epsilon_m^2 s = t_m$,
and comparing~\eqref{barom_rep} with~\eqref{barom_fin}, we conclude that
$h_* = G$, which is the desired result.
\end{proof}

\subsection{Short time asymptotics}\label{sec44}

The goal of this section is to establish the short time estimate
\eqref{short-time}. Let $\chi : [0,\infty) \to [0,1]$ be a smooth
nonincreasing function such that $\chi(x) = 1$ for $x \in [0,1/4]$ and
$\chi(x) = 0$ for $x \ge 1/2$. We define
\begin{equation}\label{f0-def}
  f_0(R,Z,t) \,=\, G(R,Z) \chi(\epsilon^2(R^2{+}Z^2))\,, \qquad
  (R,Z) \in \R^2\,, \quad t \in (0,T)\,,
\end{equation}
where $\epsilon = \sqrt{\nu t}/\bar r$ and $G(R,Z) = (4\pi)^{-1}
e^{-(R^2+Z^2)/4}$, see \eqref{RZeps-def}, \eqref{wGdef}. Due to the localization
function $\chi$, it is clear that $f_0(R,Z,t)$ vanishes when
$\epsilon R < -1/\sqrt{2}$. In particular, $f_0$ satisfies the Dirichlet
boundary condition in the time-dependent domain $\Omega_\epsilon =
\{(R,Z) \in \R^2\,|\, 1+\epsilon R > 0\}$.

If $f$ is the solution of \eqref{feq} given by \eqref{f-def}, and if
$U^\epsilon$ is the associated velocity field, we decompose
\begin{equation}\label{fUdecomp}
  \left\{\hspace{-3pt}\begin{array}{l}
  \hspace{6pt}f(R,Z,t) \,=\, f_0(R,Z,t) + \tilde f(R,Z,t)\,, \\[1mm]
  U^\epsilon(R,Z,t) \,=\, U_0^\epsilon(R,Z,t) + \tilde U^\epsilon(R,Z,t)\,,
  \end{array}\right.
  \qquad (R,Z) \in \Omega_\epsilon\,, \quad t \in (0,T)\,,
\end{equation}
where $U_0^\epsilon = \BS^\epsilon[f_0]$ and $\tilde U^\epsilon = \BS^\epsilon[
\tilde f]$. The equation satisfied by the perturbation $\tilde f$
is
\begin{equation}\label{tildefeq}
  t \partial_t \tilde f + \gamma \div_* \bigl(U_0^\epsilon \tilde f +
  \tilde U^\epsilon f_0\bigr) + \gamma\div_* \bigl(\tilde U^\epsilon
  \tilde f\bigr) \,=\, \cL \tilde f + \partial_R
  \Bigl(\frac{\epsilon\tilde f}{1+\epsilon R}\Bigr) + \cH\,,
\end{equation}
where $\cH$ is a source term which quantifies by how much $f_0$
fails to be an exact solution of \eqref{feq}. Explicitly,
\begin{equation}\label{HHdef}
  \cH \,=\, \cL f_0 + \partial_R \Bigl(\frac{\epsilon f_0}{1+\epsilon R}
  \Bigr) -t \partial_t f_0 -\gamma\div_* \bigl(U_0^\epsilon f_0\bigr)\,.
\end{equation}
Here and in what follows, if $V = (V_r,V_z)$ is a vector field
on $\Omega_\epsilon$ or on the whole plane $\R^2$, we denote $\div_* V =
\partial_R V_r + \partial_Z V_z$. Note that the perturbation
$\tilde f$ still satisfies the Dirichlet boundary condition on
$\partial\Omega_\epsilon$.

It is clear from definition \eqref{f0-def} that $f_0$ belongs for all
times to the space $X$ introduced in \eqref{Xdef}, and that $\|f_0(t)
- G\|_X \to 0$ as $t \to 0$. Thus the perturbation $\tilde f$
(implicitly extended by zero outside $\Omega_\epsilon$) belongs to $X$
for all $t \in (0,T)$, and Proposition~\ref{alpha-prop} implies that
$\|\tilde f(t)\|_X \to 0$ as $t \to 0$. In the rest of this section,
using appropriate energy estimates, we prove that $\|\tilde f(t)\|_X =
\cO(\epsilon| \log\epsilon|)$ as $t \to 0$, and this implies
\eqref{short-time} in view of the continuous injection $X
\hookrightarrow L^1(\R^2)$.

For any $t \in (0,T)$, we define
\begin{equation}\label{Edef}
  E(t) \,=\, \frac12 \int_{\Omega_\epsilon} \tilde f(R,Z,t)^2
  w(R,Z)\dd R\dd Z\,,
\end{equation}
where $w(R,Z) = e^{(R^2+Z^2)/4}$, see \eqref{wGdef}. Although
the integral in \eqref{Edef} is taken over the time-dependent domain
$\Omega_\epsilon$, there is no contribution from the boundary when
we differentiate with respect to time, because $\tilde f$ satisfies
the homogeneous Dirichlet condition on $\partial\Omega_\epsilon$.
Using \eqref{tildefeq}, we thus obtain
\begin{align}\nonumber
  t E'(t) \,&=\, \int_{\Omega_\epsilon} \tilde f(R,Z,t)
  \bigl(t\partial_t \tilde f(R,Z,t)\bigr)w(R,Z)\dd R\dd Z \\ \label{Ediff}
  \,&=\, D_1(t) + D_2(t) + H(t) - \gamma\Bigl(A_1(t) +
  A_2(t) + N(t)\Bigr)\,,
\end{align}
where
\begin{align*}
  D_1(t) \,&=\, \int_{\Omega_\epsilon} \tilde f \,(\cL \tilde f)w\dd R \dd Z\,,
  \hspace{50pt} D_2(t) \,=\, \int_{\Omega_\epsilon} \tilde f\, \partial_R
  \Bigl(\frac{\epsilon\tilde f}{1+\epsilon R}\Bigr)w\dd R \dd Z\,, \\
  A_1(t) \,&=\, \int_{\Omega_\epsilon} \tilde f \,\div_* \bigl(U_0^\epsilon \tilde f
  \bigr)w\dd R \dd Z\,, \qquad
  A_2(t) \,=\, \int_{\Omega_\epsilon} \tilde f \,\div_* \bigl(\tilde U^\epsilon f_0
  \bigr)w\dd R \dd Z\,, \\
  H(t) \,&=\, \int_{\Omega_\epsilon} \tilde f \,\cH w\dd R \dd Z\,, \hspace{66pt}
  N(t) \,=\, \int_{\Omega_\epsilon} \tilde f \,\div_* \bigl(\tilde U^\epsilon
  \tilde f\bigr)w\dd R \dd Z\,.
\end{align*}
The main result of this section is

\begin{prop}\label{energy-prop}
There exists $\delta > 0$ and, for any $\gamma > 0$, there
exist $\epsilon_0 \in (0,1/2)$ and $\kappa > 0$ such that,
if $t > 0$ is small enough so that $\epsilon \le \epsilon_0$, then
\begin{equation}\label{energy-decay}
  tE'(t) \,\le\, -2\delta \cE(t) + \kappa \epsilon |\log \epsilon|
  \,E(t)^{1/2} + \kappa E(t)^{1/2}\cE(t) + \cR(t)\,,
\end{equation}
where $\cR(t) \le e^{-1/(36\epsilon^2)}$ and
\begin{equation}\label{EEdef}
  \cE(t) \,=\, \frac12\int_{\Omega_\epsilon}\Bigl(|\nabla \tilde f|^2 +
  (1+R^2+Z^2)\tilde f^2\Bigr)w \dd R\dd Z \,\ge\, E(t)\,.
\end{equation}
\end{prop}

\begin{proof}
We proceed in several steps.

\medskip\noindent{\bf Step 1: }{\it Control of the mass}.
For any $t \in (0,T)$ we denote
\begin{equation}\label{massdef}
  m(t) \,=\, \int_{\Omega_\epsilon} \tilde f(R,Z,t)\dd R\dd Z
  \,=\, \int_{\Omega_\epsilon} \bigl(f(R,Z,t) - f_0(R,Z,t)\bigr)
  \dd R\dd Z\,.
\end{equation}
We shall show that $m(t)$ is extremely small for short times.
Indeed, since $\int_{\R^2} G\dd R\dd Z = 1$, it follows from
definition \eqref{f0-def} that
\begin{equation}\label{mass1}
  0 \,\le\, 1 - \int_{\Omega_\epsilon} f_0(R,Z,t)\dd R\dd Z \,=\,
  \int_{\R^2} \Bigl(1-\chi(\epsilon^2(R^2{+}Z^2))\Bigr)G \dd R
  \dd Z \,\le\, e^{-1/(16\epsilon^2)}\,,
\end{equation}
because, in the last integral, the integrand vanishes when
$R^2 + Z^2 \le 1/(4\epsilon^2)$. On the other hand, estimate
\eqref{flowerbound} (where one can take $\rho = \bar r/2$)
shows that
\[
  \Gamma \,\ge\, \int_{\bar r/2}^\infty\biggl\{\int_\R \omega_\theta(r,z,t)
  \dd z\biggr\}\dd r \,\ge\, \Gamma \Bigl(1 - e^{-\bar r^2/(64\nu t)}\Bigr)\,,
\]
and in view of \eqref{f-def} this implies that
\begin{equation}\label{mass2}
  0 \,\le\, 1 - \int_{\Omega_\epsilon} f(R,Z,t)\dd R\dd Z \,\le\,
  e^{-1/(64\epsilon^2)}\,.
\end{equation}
Combining \eqref{mass1} and \eqref{mass2}, we deduce that
\begin{equation}\label{massbound}
  |m(t)| \,\le\, e^{-1/(64\epsilon^2)}\,, \qquad \hbox{where}\quad
  \epsilon \,=\, \frac{\sqrt{\nu t}}{\bar r}\,.
\end{equation}

\medskip\noindent{\bf Step 2: }{\it The diffusive terms}.  After this
preliminary step, we estimate separately the various terms in the
right-hand side of \eqref{Ediff}, starting with $D_1(t)$ and $D_2(t)$
which originate from the diffusion operator in \eqref{tildefeq}. Using
the identity $(\cL\tilde f)w \,=\, \div_*(w \nabla\tilde f) + w \tilde f$
and integrating by parts, we first obtain
\begin{equation}\label{D1}
  D_1 \,=\, \int \tilde f (\cL\tilde f) w \dd R \dd Z \,=\,
  \int \Bigl(-|\nabla \tilde f|^2 + \tilde f^2\Bigr)w \dd R\dd Z\,.
\end{equation}
Here and in what follows, all integrals are taken over the domain
$\Omega_\epsilon$, or over the whole plane $\R^2$ if one extends the
integrands by zero outside $\Omega_\epsilon$ (as we implicitly do when
necessary). For simplicity we also write $\tilde f$ instead of $\tilde
f(R,Z,t)$, and similarly for other quantities.

Estimate \eqref{D1} is not sufficient for our purposes, because
it is not clear if the right-hand side is negative. To improve it, we
observe that $\tilde f (\cL\tilde f) w = \tilde g (L\tilde g)$ where
$\tilde g = w^{1/2} \tilde f$ and $L$ is the linear operator defined
by
\[
  L \tilde g \,=\, \Delta \tilde g - \frac{R^2+Z^2}{16}\tilde g
  + \frac12 \tilde g\,.
\]
We thus have the alternative formula
\begin{equation}\label{D2}
  D_1 \,=\, \int \tilde g(L\tilde g)\dd R\dd Z \,=\, \int
  \Bigl(-|\nabla \tilde g|^2 - \frac{R^2+Z^2}{16}\tilde g^2 +
  \frac12 \tilde g^2\Bigr)\dd R\dd Z\,.
\end{equation}
The operator $L$ is related to the quantum harmonic oscillator in
$\R^2$. With the normalization above, it is self-adjoint in
$L^2(\R^2)$ with spectrum $\sigma(L) = \{-n/2\,|\, n = 0,1,2\dots\}$,
and this observation already implies that $D_1 \le 0$. Moreover, the
kernel of $L$ is one-dimensional and spanned by the function
$w^{1/2}G$.  As a consequence, if $\tilde f$ has zero mean over
$\R^2$, then $\tilde g = w^{1/2} \tilde f$ is orthogonal to $w^{1/2}G$
in $L^2(\R^2)$, hence belongs to the invariant subspace where
$L \le -1/2$. Thus
\begin{equation}\label{D3}
  D_1 \,=\, \int \tilde f (\cL\tilde f) w \dd R \dd Z \,\le\,
  -\frac12 \int \tilde f^2 w \dd R\dd Z\,, \qquad \hbox{if}\quad
  \int \tilde f \dd R\dd Z \,=\, 0\,.
\end{equation}
In the general case, we can decompose $\tilde f = m(t)G + \hat f$,
so that $\hat f$ has zero mean by construction. As $\cL G = 0$,
we have $\int \tilde f (\cL\tilde f) w \dd R \dd Z = \int
\hat f (\cL\hat f) w \dd R \dd Z$ and applying \eqref{D3} to
$\hat f$ we obtain
\begin{equation}\label{D4}
   D_1 \,=\, \int \hat f (\cL\hat f) w \dd R \dd Z
  \,\le\, -\frac12 \int \hat f^2 w \dd R\dd Z \,=\,
  -\frac12 \int \tilde f^2 w \dd R\dd Z + \frac{m(t)^2}{4\pi}\,.
\end{equation}
We now take a convex combination of estimates \eqref{D1}, \eqref{D2},
and \eqref{D4}, for instance with coefficients $1/6$, $1/6$, and $2/3$.
This gives our improved bound
\begin{equation}\label{D5}
  D_1 \,\le\, -\int\Bigl(\frac16 |\nabla \tilde f|^2 +
  \frac{R^2+Z^2}{96}\tilde f^2 + \frac{1}{12}\tilde f^2\Bigr)
  w \dd R\dd Z + \frac{m(t)^2}{6\pi} \,\le\, -\frac{\cE}{48}
  + \frac{m(t)^2}{6\pi}\,.
\end{equation}

Next, we consider the second diffusive term $D_2$. Integrating
by parts and using the fact that  $\partial_R w = R w/2$, we
find
\begin{align}\nonumber
  D_2 \,&=\, -\frac{\epsilon}{2}\int \Bigl(\frac{\tilde f}{1+
  \epsilon R}\Bigr)^2 \partial_R\bigl((1{+}\epsilon R)w\bigr)
  \dd R \dd Z \\ \label{D6}
  \,&=\, -\frac{\epsilon^2}{2}\int \Bigl(\frac{\tilde f}{1+
  \epsilon R}\Bigr)^2 w \dd R \dd Z - \frac{\epsilon}{4}
  \int \Bigl(\frac{\tilde f^2}{1+\epsilon R}\Bigr) R w\dd R \dd Z\,.
\end{align}
The last term in \eqref{D6} has no sign, but is obviously harmless
when $1 + \epsilon R \ge 1/4$. In the subdomain $\tilde\Omega_\epsilon
= \{(R,Z) \,|\, 0 < 1+\epsilon R < 1/4\}$, we can apply Young's
inequality to obtain
\begin{equation}\label{D7}
  \frac{\epsilon}{4} \int_{\tilde \Omega_\epsilon} \Bigl(\frac{\tilde f^2}{1
  +\epsilon R}\Bigr) |R| w\dd R \dd Z \,\le\, \frac{\epsilon^2}{2}
  \int_{\tilde \Omega_\epsilon} \Bigl(\frac{\tilde f}{1+\epsilon R}\Bigr)^2 w
  \dd R \dd Z + \frac{1}{32} \int_{\tilde \Omega_\epsilon} f^2
  R^2 w \dd R \dd Z\,,
\end{equation}
where we replaced $\tilde f$ by $f$ in the last integrand
because $f_0$ vanishes identically in $\tilde \Omega_\epsilon$.
Using the upper bound \eqref{Gaussf} with (for instance)
$\eta = 1/4$, we see that the last integral in \eqref{D7}
is transcendentally small. Summarizing, we have shown that
\begin{equation}\label{D8}
  D_2 \,\le\, \epsilon\int \tilde f^2 |R| w \dd R \dd Z +
  \frac{C}{\epsilon^2}\,e^{-1/(16\epsilon^2)}\,,
\end{equation}
where the constant $C > 0$ depends only on $M = \gamma$. Note that the
first term in the right-hand side of \eqref{D8} is bounded by $C
\epsilon E^{1/2}\cE^{1/2} \le C\epsilon \cE$, and can therefore be controlled
by the negative terms in \eqref{D5}, if $\epsilon$ is small enough.

\medskip\noindent{\bf Step 3: }{\it The source term}. We turn our
attention to the source term $\cH$ defined in \eqref{HHdef}. We
claim that
\begin{equation}\label{H1}
  \|\cH(t)\|_{X} \,\le\, C \epsilon + C\gamma \epsilon|\log \epsilon|\,,
\end{equation}
whenever $\epsilon \le 1/2$, where $C > 0$ is a universal constant.
To prove \eqref{H1} we consider separately the various terms in
\eqref{HHdef}. First, as $\partial_t G = \cL G = 0$, it is
straightforward to verify that both quantities $t\partial_t f_0$ and
$\cL f_0$ are transcendentally small in $X$ as $\epsilon \to 0$. Next,
since $1 + \epsilon R$ is bounded away from zero on the support of
$f_0$, it is clear that the second-term in the right-hand side of
\eqref{HHdef} is $\cO(\epsilon)$ in $X$.  So the main contribution
comes from the last term $\gamma\div_* \bigl(U_0^\epsilon f_0\bigr)$,
which requires a more careful analysis.  We recall that $U_0^\epsilon
= \BS^\epsilon[f_0]$ is the velocity field obtained from the vorticity
$f_0$ via the $\epsilon$-dependent Biot-Savart law
\eqref{Ueps-rep}. As $U_0^\epsilon$ satisfies the divergence-free
condition in \eqref{U-BS}, we have the identity
\begin{equation}\label{H2}
  \div_*\bigl(U_0^\epsilon f_0\bigr) \,=\, U_0^\epsilon \cdot \nabla
  f_0 + \div_*(U_0^\epsilon) f_0 \,=\, (U_0^\epsilon - U_0^0) \cdot \nabla
  f_0 - \frac{\epsilon U_{0,r}^\epsilon}{1+\epsilon R}f_0\,,
\end{equation}
where $U_0^0 = \BS^0[f_0]$ denotes the velocity field obtained from
$f_0$ via the two-dimensional Biot-Savart law \eqref{BS2D}. Note that
$U_0^0 \cdot \nabla f_0 = 0$, because $f_0$ is radially symmetric in
$\R^2$, but we included that term in \eqref{H2} so that the right-hand
side contains the difference $U_0^\epsilon - U_0^0$.  Again, since $1
+ \epsilon R$ is bounded away from zero on the support of $f_0$ and
$U_0^\epsilon$ is uniformly bounded by \eqref{unif-Ueps}, the last
term in \eqref{H2} is $\cO(\epsilon)$ in $X$. As for the first term,
we deduce from \eqref{Udiffest} that
\begin{equation}\label{H3}
  |U_0^\epsilon(R,Z,t) - U_0^0(R,Z,t)| \,\le\, C \frac{\epsilon}{1+
  \epsilon R}\Bigl(1 + \logp\frac{1+\epsilon R}{\epsilon}
  \Bigr)\,,
\end{equation}
and it follows easily that $\|(U_0^\epsilon - U_0^0) \cdot \nabla
f_0\|_X \le C \epsilon|\log \epsilon|$ if $\epsilon$ is small
enough. This concludes the proof of \eqref{H1}, and we deduce
that
\begin{equation}\label{H7}
   H \,=\, \int \tilde f \,\cH w\dd R \dd Z\, \,\le\,
   C E^{1/2}\Bigl(\epsilon + \gamma \epsilon|\log \epsilon|
  \Bigr)\,,
\end{equation}
whenever $\epsilon \le 1/2$.

\medskip\noindent{\bf Step 4: }{\it The advection terms}. We now
consider the terms produced by the advection operator $\tilde f
\mapsto \div_*(U_0^\epsilon \tilde f)$ and the nonlocal operator
$\tilde f \mapsto \div_*(\tilde U^\epsilon f_0)$ in \eqref{tildefeq}.
Integrating by parts, we find
\begin{equation}\label{A1}
  A_1 \,=\, \int \tilde f \,\div_* \bigl(U_0^\epsilon \tilde f
  \bigr)w\dd R \dd Z \,=\, \frac12 \int \tilde f^2\Bigl((\div_*
  U_0^\epsilon)w - \bigl((U_0^\epsilon - U_0^0)\cdot \nabla w\bigr)\Bigr)
  \dd R\dd Z\,,
\end{equation}
where $U_0^0 = \BS^0[f_0]$. Note that $U_0^0 \cdot \nabla w = 0$
because $w$ is radially symmetric, but it is useful to make the
difference $U_0^\epsilon - U_0^0$ appear in \eqref{A1}. In the region
where $1+\epsilon R \ge 1/4$, the integrand in the last member of
\eqref{A1} can be estimated using \eqref{H3} and the bound $|\div_*
U_0^\epsilon| \le C\epsilon(1+\epsilon R)^{-1}$, which follows from
the identity
\[
  \div_* U_0^\epsilon \,=\, -\frac{\epsilon U_{0,r}^\epsilon}{(1
  +\epsilon R)}\,.
\]
If $1 + \epsilon R < 1/4$, namely if $(R,Z) \in \tilde \Omega_\epsilon$,
we simply use the upper bound \eqref{Gaussf} for $f = \tilde f$ together
with uniform estimates on $U_0^\epsilon$ and $U_0^0$, which can be deduced
from \eqref{Uepsbound}. We thus obtain, whenever $\epsilon \le 1/2$,
\begin{align}\nonumber
  |A_1| \,&\le\, \int_{\Omega_\epsilon\setminus\tilde \Omega_\epsilon} \frac{C\epsilon}{1
  {+}\epsilon R}\Bigl(1 + \logp\frac{1{+}\epsilon R}{\epsilon}\Bigr)
  \tilde f^2 (1{+}|R|{+}|Z|)w \dd R\dd Z + \int_{\tilde\Omega_\epsilon}
  f^2 (1{+}|R|{+}|Z|)w\dd R\dd Z \\ \nonumber
  \,&\le\, C\epsilon|\log\epsilon| \int \tilde f^2 (1{+}|R|{+}|Z|)w\dd R\dd Z
  + \frac{C}{\epsilon}\,e^{-1/(16\epsilon^2)} \\ \label{A2}
  \,&\le\, C\epsilon|\log\epsilon| \,E^{1/2}\cE^{1/2} +
  \frac{C}{\epsilon}\,e^{-1/(16\epsilon^2)}\,.
\end{align}
Note that the quantity $C\epsilon|\log\epsilon| \,E^{1/2}\cE^{1/2}$ can be 
controlled by the negative terms in \eqref{D5}, if $\epsilon$ is sufficiently 
small.

As for the nonlocal term $A_2$, we observe that
\begin{equation}\label{A3}
   A_2 \,=\, \int \tilde f \,\div_* \bigl([\tilde U^\epsilon -
  \tilde U^0]f_0\bigr)w\dd R \dd Z + \int \tilde f\,
  (\tilde U^0 \cdot \nabla [f_0 - G]) w\dd R \dd Z\,,
\end{equation}
where $\tilde U^0 = \BS^0[\tilde f]$ is the velocity field obtained
from the vorticity $\tilde f$ via the two-dimensional Biot-Savart
law \eqref{BS2D}. In deriving \eqref{A3} we used the nontrivial
observation
\[
  \int \tilde f\,\bigl(\tilde U^0 \cdot \nabla G\bigr) w\dd R \dd Z
  \,\equiv\, \int_{\R^2} \tilde f\,\bigl(\BS^0[\tilde f]\cdot \nabla G\bigr)
  w\dd R \dd Z \,=\, 0\,,
\]
which was first made in \cite[Lemma~4.8]{GW2}. Let $A_{21}$ denote the
first term in the right-hand side of \eqref{A3}. Integrating by
parts, we find
\begin{equation}\label{A4}
  A_{21} \,=\, -\int f_0\, \bigl(\tilde U^\epsilon - \tilde U^0\bigr)
  \cdot \nabla (\tilde f w) \dd R \dd Z\,.
\end{equation}
Note once again that $f_0$ is supported in the region where $1 +
\epsilon R \ge 1/4$, and in that domain we infer from \eqref{Udiffest}
that $|\tilde U^\epsilon - \tilde U^0| \le C\epsilon|\log\epsilon|
\|\tilde f\|_{L^1\cap L^2}$. Using H\"older's inequality and the continuous
injection $X \hookrightarrow L^1(\R^2) \cap L^2(\R^2)$, we deduce that
\begin{align}\nonumber
  |A_{21}| \,&\le\,  C\epsilon|\log\epsilon| \|\tilde f\|_X
  \int \Bigl(|\nabla\tilde f| + (|R|^2+|Z^2|)^{1/2}|\tilde f|
  \Bigr)\dd R\dd Z\, \\ \nonumber
  \,&\le\, C\epsilon|\log\epsilon| E^{1/2} \biggl\{\int
  \Bigl(|\nabla\tilde f|^2 + (|R|^2+|Z^2|)|\tilde f|^2\Bigr)
  w\dd R\dd Z\, \biggr\}^{1/2} \\ \label{A5}
  \,&\le\, C\epsilon|\log\epsilon| E^{1/2} \cE^{1/2}\,.
\end{align}
Again the right-hand side can be controlled by the negative terms
in \eqref{D5} if $\epsilon$ is sufficiently small.

Finally, let $A_{22}$ denote the last integral term in \eqref{A3}.
Here the integral is taken over the domain $\hat \Omega_\epsilon =
\{(R,Z)\,|\, \epsilon^2(R^2+Z^2) \ge 1/4\}$, because $f_0 = G$
on $\Omega_\epsilon \setminus \hat \Omega_\epsilon$. Using H\"odler's
inequality, we obtain
\[
  [A_{22}| \,\le\, C \int_{\hat\Omega_\epsilon} |\tilde U^0|
  |\tilde f| (R^2{+}Z^2)^{1/2} \dd R\dd Z \,\le\,
  C \|\tilde U^0\|_{L^4} \biggl\{\int_{\hat\Omega_\epsilon}
  |\tilde f|^{4/3} (R^2{+}Z^2)^{2/3}\dd R\dd Z\biggr\}^{3/4}\,.
\]
As $\tilde U^0$ is the velocity field obtained from $\tilde f$ via the
two-dimensional Biot-Savart law, the Hardy-Littlewood-Sobolev inequality
implies that $\|\tilde U^0\|_{L^4} \le C \|\tilde f\|_{L^{4/3}} \le C
\|\tilde f\|_X$. On the other hand, using H\"older's inequality
again, we find
\[
  \int_{\hat\Omega_\epsilon}|\tilde f|^{4/3} (R^2{+}Z^2)^{2/3}\dd R\dd Z
  \,\le\, \biggl\{\int |\tilde f|^2 w\dd R\dd Z\biggr\}^{2/3}
  \biggl\{\int_{\hat\Omega_\epsilon} (R^2{+}Z^2)\frac{1}{w^2}
  \dd R\dd Z\biggr\}^{1/3}\,,
\]
where the last integral can be explicitly computed and is
found to be transcendentally small as $\epsilon \to 0$.
Altogether, we have shown that
\begin{equation}\label{A6}
  |A_{22}| \,\le\, \frac{C}{\epsilon^{1/2}}\,e^{-1/(32\epsilon^2)}\,
  \|\tilde U^0\|_{L^4}  \|\tilde f\|_X \,\le\, \frac{C}{\epsilon^{1/2}}
  \,e^{-1/(32\epsilon^2)} \|\tilde f\|_X^2\,\,.
\end{equation}

\medskip\noindent{\bf Step 5: }{\it The nonlinear term}.
Finally we consider the nonlinear term $N$ is \eqref{Ediff}.
Integrating by part, we find
\[
  N(t) \,=\, -\int \tilde f\,\tilde U^\epsilon \cdot \bigl(w \nabla
  \tilde f + \tilde f \nabla w\bigr)\dd R \dd Z\,,
\]
so that
\[
  |N(t)| \,\le\, C \int |\tilde U^\epsilon|\, \bigl(|\tilde f|
  w^{1/2}\bigr)\,\Bigl(\bigl(|\nabla\tilde f| + (R^2{+}Z^2)^{1/2}
  |\tilde f|\bigr)w^{1/2}\Bigr)\dd R \dd Z\,.
\]
We apply the trilinear H\"older inequality to the right-hand side,
with exponents $4$, $4$, and $2$. Since $\tilde U^\epsilon =
\BS^\epsilon[\tilde f]$, it follows from \eqref{Uepsbound} (using
again the Hardy-Littlewood-Sobolev inequality) that $\|\tilde
U^\epsilon\|_{L^4} \le C\|\tilde f\|_{L^{4/3}} \le C\|\tilde
f\|_X$. On the other hand, using Sobolev's interpolation
inequality, we see that
\[
  \|\tilde f w^{1/2}\|_{L^4}^2 \,\le\, C  \|\tilde f w^{1/2}\|_{L^2}
  \|\nabla(\tilde f w^{1/2})\|_{L^2} \,\le\, C E^{1/2}\cE^{1/2}\,.
\]
Finally,
\[
  \|\bigl(|\nabla\tilde f| + (R^2{+}Z^2)^{1/2}|\tilde f|\bigr)w^{1/2}
  \|_{L^2} \,\le\, C \cE^{1/2}\,.
\]
Altogether, we have shown that
\begin{equation}\label{N1}
  |N| \,\le\, C E^{3/4} \cE^{3/4} \,\le\, C E^{1/2}\cE~.
\end{equation}
Alternatively, one can apply the trilinear H\"older inequality 
with exponents $\infty$, $2$, $2$, and deduce from \eqref{Uepsbound} 
that $\|\tilde U^\epsilon\|_{L^\infty} \le C \|\tilde f\|_{L^{4/3}}^{1/2}
\|\tilde f\|_{L^4}^{1/2} \le C E^{1/4}\cE^{1/4}$. This also leads 
to \eqref{N1}. 

\medskip\noindent{\bf Step 6: }{\it Conclusion}.
Combining estimates \eqref{massbound}, \eqref{D5}, \eqref{D8},
\eqref{H7}, \eqref{A2}, \eqref{A5}, \eqref{A6}, and \eqref{N1},
we obtain \eqref{Ediff}.
\end{proof}

\begin{proof}[\bf Proof of estimate \eqref{short-time} in
Theorem~\ref{main}.] We know from Proposition~\ref{alpha-prop}
that $f(t)$ converges to $G$ in $X$ as $t \to 0$, and so does
$f_0(t)$ in view of definition \eqref{f0-def}. Thus $E(t) \to
0$ as $t \to 0$. As long as $t$ is small enough so that
$\epsilon \le \epsilon_0$ and $\kappa E(t)^{1/2} \le \delta/2$,
it follows from \eqref{energy-decay} and Young's inequality
that
\begin{equation}\label{ener1}
  tE'(t) \,\le\, -\delta \cE(t) + \cR_1(t) \,\le\,
  -\delta E(t) + \cR_1(t)\,,
\end{equation}
where $\cR_1 \le C \epsilon^2 |\log \epsilon|^2$.
Integrating that differential inequality, we obtain
\begin{equation}\label{ener2}
  E(t) \,\le\, t^{-\delta}\int_0^t s^{\delta-1} \cR_1(s)\dd s
 \,=:\, \cR_2(t)\,,
\end{equation}
where again $\cR_2 \le C \epsilon^2 |\log \epsilon|^2$.
From that bound, we see that there exists $\epsilon_1 \in (0,
\epsilon_0)$ such that our assumption $\kappa E^{1/2} \le \delta/2$
is satisfied whenever $\epsilon \le \epsilon_1$. So, for $\epsilon
\le \epsilon_1$, we have
\[
  \|f(t) - f_0(t)\|_{L^1(\Omega_\epsilon)} \,=\, \|\tilde
  f(t)\|_{L^1(\Omega_\epsilon)} \,\le\, C \|\tilde f(t)\|_X \,\le\,
  C E(t)^{1/2} \,\le\, C \epsilon |\log \epsilon|\,,
\]
and since $f_0$ is extremely close to $G$ this proves exactly
\eqref{short-time}, after returning to the original variables. When
$\epsilon_1 \le \epsilon \le 1/2$, estimate \eqref{short-time}
obviously holds (for some appropriate constant $C_1$), because the
left-hand side is trivially smaller than $2|\Gamma|$.
\end{proof}

\subsection{Uniqueness}\label{sec45}

This final section is devoted to the uniqueness claim in Theorem~\ref{main}.
Assume for this purpose that $\omega_\theta^{(1)}, \omega_\theta^{(2)} \in
C^0((0,T),L^1(\Omega) \cap L^\infty(\Omega))$ are two mild solutions of
equation~\eqref{omeq} which are uniformly bounded in $L^1(\Omega)$
and converge weakly to $\Gamma \,\delta_{(\bar r,\bar z)}$ as $t \to 0$.
Introducing self-similar variables as in \eqref{f-def}, we obtained
two rescaled vorticities $f_1(R,Z,t), f_2(R,Z,t)$ which can both
be decomposed as in \eqref{fUdecomp}:
\[
  f_1(R,Z,t) \,=\, f_0(R,Z,t) + \tilde f_1(R,Z,t)\,, \qquad
  f_2(R,Z,t) \,=\, f_0(R,Z,t) + \tilde f_2(R,Z,t)\,.
\]
The associated velocity fields are decomposed in a similar way:
\[
  U_1^\epsilon(R,Z,t) \,=\, U_0^\epsilon(R,Z,t) + \tilde U_1^\epsilon(R,Z,t)\,,
  \qquad
  U_2^\epsilon(R,Z,t) \,=\, U_0^\epsilon(R,Z,t) + \tilde U_2^\epsilon(R,Z,t)\,.
\]
We take the difference of both solutions and denote
\begin{align*}
  \tilde f(R,Z,t) \,&=\, f_1(R,Z,t) - f_2(R,Z,t) \,=\,
  \tilde f_1(R,Z,t) - \tilde f_2(R,Z,t)\,, \\
  \tilde U^\epsilon(R,Z,t) \,&=\, U_1^\epsilon(R,Z,t) - U_2^\epsilon(R,Z,t)
  \,=\, \tilde U_1^\epsilon(R,Z,t) - \tilde U_2^\epsilon(R,Z,t)\,.
\end{align*}
The evolution equation for $\tilde f$ reads
\begin{equation}\label{tildefeq2}
  t \partial_t \tilde f + \gamma \div_* \bigl(U_0^\epsilon \tilde f +
  \tilde U^\epsilon f_0\bigr) + \gamma\div_* \bigl(\tilde U^\epsilon
  \tilde f_1 + \tilde U_2^\epsilon\tilde f \bigr) \,=\, \cL \tilde f
  + \partial_R \Bigl(\frac{\epsilon\tilde f}{1+\epsilon R}\Bigr)\,.
\end{equation}
This is basically the same equation as \eqref{tildefeq}, except that
the source term $\cH$ has disappeared when taking the difference of
the equations for $\tilde f_1$ and $\tilde f_2$, and the nonlinear
term has been expanded as follows: $\tilde U_1^\epsilon \tilde f_1 -
\tilde U_2^\epsilon \tilde f_2 = \bigl(\tilde U_1^\epsilon -
\tilde U_2^\epsilon\bigr) \tilde f_1 + \tilde U_2^\epsilon
\bigl(\tilde f_1 - \tilde f_2\bigr)$. In analogy with \eqref{Edef}
we denote
\[
  E\,=\, \frac12 \int \tilde f^2 w\dd R\dd Z\,, \qquad
  E_1\,=\, \frac12 \int \tilde f_1^2 w\dd R\dd Z\,, \qquad
  E_2\,=\, \frac12 \int \tilde f_2^2 w\dd R\dd Z\,,
\]
and as in \eqref{EEdef} we also define
\[
  \cE \,=\, \frac12 \Bigl(|\nabla \tilde f|^2 +
  (1+R^2+Z^2)\tilde f^2\Bigr)w \dd R\dd Z\,.
\]

Now, repeating the proof of Proposition~\ref{energy-prop}, we first obtain
the estimate
\begin{equation}\label{energy-decay2}
  tE'(t) \,\le\, -2\delta \cE(t) + \kappa \bigl(E_1(t)^{1/2} + E_2(t)^{1/2}
  \bigr) \cE(t) + \tilde\cR(t)\,,
\end{equation}
which holds for $\epsilon \le \epsilon_0 \le 1/2$ with a remainder
term satisfying $\tilde\cR(t) \le e^{-1/(36\epsilon^2)}$. Note that
\eqref{energy-decay2} does not include the term $\kappa \epsilon |\log
\epsilon| E(t)^{1/2}$ which, in the corresponding estimate
\eqref{energy-decay}, was produced by the source term $\cH$.
As long as $t$ is small enough so that $\kappa (E_1(t)^{1/2} +
E_2(t)^{1/2}) \le \delta$, it follows from \eqref{energy-decay2}
that $tE'(t) \le -\delta E(t) + \tilde\cR(t)$, hence
\begin{equation}\label{firstbdd}
  E(t) \,\le\, t^{-\delta}\int_0^t s^{\delta-1} \tilde\cR(s)\dd s
  \,=\, \cO\Bigl(e^{-1/(36\epsilon^2)}\Bigr)\,.
\end{equation}
This already shows that $E(t)$ converges extremely rapidly to
zero as as $t \to 0$, but our actual goal is to prove that $E(t)$
vanishes identically.

To do that, we combine \eqref{energy-decay2} with another estimate,
which is less sophisticated and easier to establish. As long as
$\epsilon \le 1/2$, we claim that
\begin{equation}\label{energy-decay3}
  tE'(t) \,\le\, -\delta \cE(t) + K E(t) + \kappa \bigl(E_1(t)^{1/2}
  + E_2(t)^{1/2}\bigr) \cE(t)\,,
\end{equation}
for some positive constants $K$ and $\kappa$ (depending on $\gamma$).
Note that there is no ``inhomogeneous'' term $\tilde\cR(t)$ in
\eqref{energy-decay3}, but this is obtained at the expense of
including the positive term $K E(t)$ with a (possibly large)
constant $K$. To obtain \eqref{energy-decay3}, the main modification
in the proof of Proposition~\ref{energy-prop} concerns the diffusive
terms $D_1$ and $D_2$. To bound $D_1$ we forget about \eqref{D4}
and only take a convex combination of \eqref{D1}, \eqref{D2},
with coefficients $1/3$ and $2/3$. The result is
\begin{equation}\label{D1new}
  D_1 \,\le\, -\int \Bigl(\frac13 |\nabla\tilde f|^2 + \frac{R^2{+}Z^2}{24}
  \tilde f^2 + \frac13 \tilde f^2\Bigr)\dd R\dd Z +
  \int \tilde f^2 w \dd R\dd Z\,.
\end{equation}
As for $D_2$, we use estimate \eqref{D7} on the whole domain
$\Omega_\epsilon$ and add it to \eqref{D6}, which gives
\begin{equation}\label{D2new}
   D_2 \,\le\, \frac{1}{32} \int \tilde f^2 R^2 w \dd R \dd Z\,.
\end{equation}
When taking the sum $D_1 + D_2$, we observe that the right-hand side
of \eqref{D2new} is entirely absorbed in the negative terms that
appear in \eqref{D1new}. On the other hand, the advection terms $A_1$
and $A_2$ can be bounded in a crude way, because we do not need to use
subtle cancellations to produce a factor of $\epsilon$. For instance,
one can verify using \eqref{Uepsbound}, \eqref{Ueps-rep} that the 
quantities $U_0^\epsilon$ and $\div_*(U_0^\epsilon)$ are bounded on $\Omega_\epsilon$, 
hence it follows from \eqref{A1} that $|A_1| \le C \int \tilde f^2 \dd R\dd Z$ 
for some positive constant $C$. In view of \eqref{A5} and \eqref{A6}, 
a similar estimate holds for $A_2$, too. Finally the nonlinear terms are 
treated as in \eqref{N1}, and we arrive at \eqref{energy-decay3}.

Now, whenever $t$ is small enough so that $\kappa (E_1(t)^{1/2} + E_2(t)^{1/2})
\le \delta$, it follows from \eqref{energy-decay3} that $tE'(t) \le
K E(t)$, hence
\begin{equation}\label{secondbdd}
  E(t) \,\le\, \Bigl(\frac{t}{t_0}\Bigr)^K E(t_0)\,, \qquad 0 < t_0 < t\,.
\end{equation}
In view of \eqref{firstbdd}, the right-hand side of \eqref{secondbdd}
converges to $0$ as $t_0 \to 0$. Thus $E(t) = 0$, and we deduce that
$f_1(t) = f_2(t)$ for all sufficiently small times. Returning to
the original variables, we conclude that
\[
  \omega_\theta^{(1)}(r,z,t) \,=\, \omega_\theta^{(2)}(r,z,t)\,,
\]
for sufficiently small times, hence for all $t \in (0,T)$ in
view of the well-posedness result established in \cite[Theorem~1.1]{GS}.
The proof of Theorem~\ref{main} is now complete. \QED

\section{Appendix}\label{sec5}

\subsection{Convergence of signed measures}
\label{sec51}

For easy reference, we collect here a few remarks on weak convergence
of signed measures. The content of this section is probably standard,
although most of the classical literature is devoted to the particular
case of probability measures. We state the results in a general
framework, but in the rest of the paper all measures are defined
on the half-plane $\Omega \subset \R^2$. We first recall a few definitions.

\smallskip\noindent{\bf 1.}
Given a locally compact metric space
$X$, we denote by $C_0(X)$ the space of all continuous functions
$f : X \to \R$ that vanish at infinity in the following sense:
for any $\epsilon > 0$, there exists a compact set $K \subset X$
such that $|f(x)| \le \epsilon$ for all $x \in K^c := X\setminus K$.
Equipped with the supremum norm, $C_0(X)$ is a real Banach space.

\smallskip\noindent{\bf 2.}
Let $\cM(X)$ be the set of all finite, signed, regular Borel measures
on $X$. If $\mu \in X$, we denote by $|\mu|$ the total variation of
$\mu$ \cite{Ru}, which is a nonnegative finite Borel measure on
$X$. The total variation norm of $\mu$ is the real number $\|\mu\|
= |\mu|(X) \ge 0$. Equipped with the total variation norm, the
space $\cM(X)$ becomes a real Banach space.

\smallskip\noindent{\bf 3.}
By the Riesz-Markov theorem \cite{Ru}, if $\Phi : C_0(X) \to \R$ is
any continuous linear functional, there exists a unique measure $\mu
\in \cM(X)$ such that
\begin{equation}\label{Riesz-Markov}
  \Phi(f) \,=\, \int_X f \dd\mu\,, \qquad \hbox{for all }
  f \in C_0(X)\,.
\end{equation}
Moreover the total variation norm $\|\mu\|$ is precisely the norm of
the linear functional $\Phi$. The space $\cM(X)$ can thus be
identified via~\eqref{Riesz-Markov} to the topological dual
$C_0(X)'$.

\smallskip\noindent{\bf 4.}
If $(\mu_n)$ is a sequence in $\cM(X)$, we say that $\mu_n$
{\em converges weakly} to $\mu \in \cM(X)$ if
\begin{equation}\label{weak-def}
  \lim_{n \to \infty} \int_X f \dd\mu_n \,=\, \int_X f \dd\mu\,,
  \qquad \hbox{for all } f \in C_0(X)\,.
\end{equation}
We write $\mu_n \weakto \mu$ as $n \to \infty$. This notion
coincides with the weak-$*$ convergence in $\cM(X) \simeq
C_0(X)'$. We always have
\[
  \|\mu\| \,\le\, \liminf_{n \to \infty} \|\mu_n\|\,.
\]

\smallskip\noindent{\bf 5.} A family of measures $\cF \subset \cM(X)$
is {\em tight} if, for any $\epsilon > 0$, there exists a compact set
$K \subset X$ such that $|\mu|(K^c) \le \epsilon$ for all $\mu \in
\cF$. Any singleton $\{\mu\}$ is necessarily tight, because the
measure $\mu \in \cM(X)$ is inner regular. If $(\mu_n)$ is a tight
sequence in $\cM(X)$ that converges weakly to $\mu \in \cM(X)$,
the convergence in~\eqref{weak-def} holds for all bounded and
continuous functions $f : X \to \R$, and not only for all
$f \in C_0(X)$. This is the case, for instance, if $(\mu_n)$
is a sequence of probability measures that converges to a
probability measure $\mu$.

The main purpose of this section is to state the following
basic result:

\begin{prop}\label{weak-prop}
Let $(\mu_n)$ be a sequence in $\cM(X)$, and let $\mu \in \cM(X)$.
We assume that
\[
  \mu_n \,\weakto\, \mu \quad \hbox{and}\quad
  \|\mu_n\| \,\to\, \|\mu\|\,, \qquad \hbox{as }\,n \to \infty\,.
\]
Then $|\mu_n| \weakto |\mu|$ as $n \to \infty$, and the sequence
$(|\mu_n|)$ is tight.
\end{prop}

\begin{proof}
The result is obvious if $\mu = 0$, so we assume henceforth that
$\mu \neq 0$. Then $\mu_n \neq 0$ for all sufficiently large $n \in
\N$, and we can thus define the normalized measures
$$
  \tilde \mu_n \,=\, \frac{\mu_n}{\|\mu_n\|}~, \qquad
  \hbox{and}\quad \tilde \mu \,=\, \frac{\mu}{\|\mu\|}~.
$$
By construction $|\tilde \mu_n|$ and $|\tilde \mu|$ are now
probability measures on $X$, and $\tilde \mu_n \weakto \tilde\mu$
as $n \to \infty$.

Let $U$ be an open subset of $X$, and take $f \in C_0(U)$ such that
$|f(x)| \le 1$ for all $x \in U$. We denote by $\bar f : X \to \R$
the extension of $f$ by zero outside $U$. One verifies that
$\bar f \in C_0(X)$, so that
$$
  \Bigl|\int_U f \dd\tilde\mu\Bigr| \,=\, \Bigl|\int_X \bar f
  \dd\tilde\mu\Bigr| \,=\,  \lim_{n\to\infty}\Bigl|\int_X \bar f
  \dd\tilde\mu_n\Bigr| \,\le\, \liminf_{n \to \infty}|\tilde\mu_n|(U)~,
$$
because $|\bar f| \le {\rm 1}_U$ (the indicator function of $U$).
It follows that
$$
  |\tilde \mu|(U) \,=\, \sup\left\{ \Bigl|\int_U f \dd\tilde\mu\Bigr| ~;\,
  f \in C_0(U)\,,~\|f\|_{\infty} \le 1\right\} \,\le\, \liminf_{n \to \infty}
  |\tilde\mu_n|(U)~. \eqno(3)
$$
Since (3) holds for any open set $U \subset X$, the celebrated
Portmanteau theorem \cite{Bi} implies that $|\tilde \mu_n| \weakto
|\tilde \mu|$ as $n \to \infty$, hence also $|\mu_n| \weakto |\mu|$
as $n \to \infty$.

As $(|\tilde \mu_n|)$ is a sequence of probability measures that
converges weakly to the probability measure $|\tilde \mu| \in \cM(X)$,
the sequence $(|\tilde \mu_n|)$ is tight (see the discussion above),
and so is the sequence $(|\mu_n|)$.
\end{proof}

\subsection{Velocity bounds in $L^\infty(\R^3)^{-1}$}\label{sec52}

This section is devoted to the proof of Lemma~\ref{lemma-bmo-bound}.
We first note that it is enough to show that $\|u\|_{(L^\infty)^{-1}}
\le c$ when $\omega_\theta = \delta_{(\bar r,\bar z)}$ for some $(\bar r, 
\bar z) \in \Omega$, as the general situation can be thought of as
a continuous superposition of these special cases. Moreover, due to
the scaling invariance and the translational symmetry along the
$z$--axis, it is enough to consider the particular case where $\bar
r=1$, $\bar z=0$.

The proof can be motivated by the following observation, which is
as a variant of formula (1.11) in~\cite{Os}:
\begin{equation}\label{z1}
  \partial_i \log |x| \,=\, \div \,(x_i \nabla \log |x|)\,\hbox{ in } \R^2
  \quad (i = 1,2)\,.
\end{equation}
This shows that, in dimension two, the vector field $\nabla\log |x|$ belongs
to $(L^\infty)^{-1}$, and not only to $\BMO^{-1}$. We now consider a
three-dimensional analogue of~\eqref{z1}, which is adapted to our purposes.
Let
\[
 \cG(x) \,=\, \frac1{4\pi |x|}\,,\qquad x\in\R^3\setminus\{0\}\,,
\]
be the fundamental solution of the Laplacian in $\R^3$, and consider the
matrix-valued function
\[
  P \,=\, \begin{pmatrix}
  -\partial_3\cG & 0 & \partial_1 \cG\\
  0& -\partial_3 \cG & \partial_2 \cG\\
  \partial_1 \cG &\partial_2 \cG &\partial_3 \cG
  \end{pmatrix}\,.
\]
Note that $\div \,P = 0$ in $\R^3\setminus\{0\}$, where (as usual) $\div \,P$ is
the vector given in coordinates by $(\div\, P)_i = \partial_j P_{ij}$.
In the sense of distributions, we have
\begin{equation}\label{z6}
  \nabla \cG \,=\, \div\,(x_3 P)\,, \quad \hbox{in }\cD'(\R^3)\,.
\end{equation}

Let us parametrize the vortex filament supported by the circle $\cC =
\{(x_1,x_2,0)\,|\, x_1^2 + x_2^2 = 1\}$ using $\gamma(s) = (\cos s,
\sin s, 0)$ for $s\in(-\pi,\pi]$. The associated velocity field $U$ is
\begin{equation}\label{z8}
  U(x) \,=\, \curl \int_{-\pi}^{\pi} \cG(x-\gamma(s))\gamma'(s)\dd s
  \,=\, \int_{-\pi}^{\pi} \nabla \cG(x-\gamma(s))\wedge \gamma'(s)\dd s\,.
\end{equation}
Using~\eqref{z6} together with the fact that $|\gamma'(s)| = 1$, we see that
to prove our claim, it is enough to establish a uniform bound for the
quantity
\[
  \int_{-\pi}^{\pi}| (x-\gamma(s))_3P(x-\gamma(s))|\dd s \,=\,
  |x_3| \int_{-\pi}^{\pi} |P(x-\gamma(s))|\dd s\,.
\]
As $|P(x)|\le c |x|^{-2}$, we only need to bound the expression
\[
  I(x) \,=\, \int_{-\pi}^{\pi}\frac{|x_3|}{|x-\gamma(s)|^2}\dd s
  \,=\, \frac{2\pi|x_3|}{\sqrt{(1+|x|^2)^2 - 4(x_1^2 + x_2^2)}}\,.
\]
But $(1+|x|^2)^2 - 4(x_1^2 + x_2^2) = (1-x_1^2-x_2^2)^2 + 2(1+x_1^2+x_2^2)x_3^2 +
x_3^4 \ge 2x_3^2$, hence $I(x) \le \sqrt{2}\pi$. The proof is thus
complete. \QED

\begin{rem}\label{BB_rem}
The above argument seems to be related to results of Bourgain-Brezis
in~\cite{BB} on the solvability of the equation $\div Y=f$ in spaces
for which standard elliptic theory fails. But in the more general situation
considered in that paper we could not find exactly the estimates that
we need in the special case considered here.
\end{rem}

\begin{rem}\label{BMO_rem}
If we only wish to prove a $\BMO^{-1}$ bound for $u$, which is sufficient
to apply the results of~\cite{SSSZ}, we see from \eqref{z8} that it is
enough to estimate the vector field
\[
  A(x) \,=\, \int_{-\pi}^{\pi} \cG(x-\gamma(s))\gamma'(s)\dd s
\]
in the space $\BMO$. This can be done in a number of ways. For
example, we note that $\nabla A\in L^p(\R^3)$ for any $p\in(1,2)$, and
near the circle $\cC$ we have $|\nabla A(x)| \lesssim {\rm dist}
(x,\cC)^{-1}$.  This easily gives a uniform bound on $R^{-3+p}\int_{B_{x,R}}
|\nabla A(y)|^p\dd y$, which implies that $A\in \BMO$.

If one is willing to use deeper results in harmonic analysis, one can
apply for example Theorem 3 on page 159 of Stein's book~\cite{Stein}
and some elementary estimates to see that, for the $\BMO$ bound of
$A=\Delta^{-1}\omega$, it is enough to control
\[
  \sup_{x \in \R^3}\,\sup_{R>0} \,\,\frac1R \,\int_{B_{x,R}}|\omega(y)|\dd y\,.
\]
That quantity is in turn bounded by $c\|\omega_\theta\|_{L^1(\Omega)}$,
as is easily verified.
\end{rem}

\subsection{Bounds on the fundamental solution}
\label{sec53}

This section is devoted to the proof of Proposition~\ref{fundsol-prop}.
Since the existence of a (unique) fundamental solution $\Phi$ is known
from the work of Aronson, we concentrate on the derivation of the
upper bound~\eqref{fund-bound}, and for that purpose we adapt to
our particular situation the efficient approach of Fabes and Stroock
\cite{FaSt}. Without loss of generality, we take $\nu = 1$, we assume
that the functions $U,V$ are smooth and bounded on $\R^n \times [0,T]$,
and we prove estimate~\eqref{fund-bound} for $s = 0$.

Let $f$ be a smooth solution to~\eqref{f-evol} on $\R^n \times [0,T]$,
with (for instance) compactly supported initial data. Given any
fixed vector $\alpha \in \R^n$, we define $g(x,t) = e^{-\alpha\cdot x}
f(x,t)$ for $x \in \R^n$ and $t \in [0,T]$. The evolution
equation satisfied by $g$ is
\begin{equation}\label{g-evol}
  \partial_t g + U\cdot\nabla g + (U\cdot\alpha + V)g \,=\,
  \Delta g + 2\alpha\cdot \nabla g + \alpha^2 g\,.
\end{equation}
The proof of the upper bound on the fundamental solution
of~\eqref{f-evol} involves four steps:

\medskip\noindent{\bf Step 1:} $L^1$ estimate. Assuming first
that $g$ is a nonnegative solution of~\eqref{g-evol}, and
using the assumption that $\div U = 0$, we compute
\begin{align*}
  \frac{\D}{\D t}\int g\dd x \,&=\, \alpha^2 \int g\dd x -
  \int (U\cdot\alpha + V)g \dd x \\
  \,&\le\, \Bigl(\alpha^2 + |\alpha| \|U(t)\|_{L^\infty}
  + \|V(t)\|_{L^\infty}\Bigr)\int g\dd x\,.
\end{align*}
Here and in what follows all integrals are taken over the whole
Euclidean space $\R^n$, and for simplicity we write
$\|U(t)\|_{L^\infty}$ instead of $\|U(\cdot,t)\|_{L^\infty(\R^n)}$.
Applying Gronwall's lemma, we obtain the estimate
\begin{equation}\label{g-L1}
  \int |g(x,t)|\dd x \,\le\, \biggl(\int |g(x,0)|\dd x\biggr)
  \exp\Bigl(\alpha^2 t + \int_0^t \bigl(|\alpha| \|U(s)\|_{L^\infty}
  + \|V(s)\|_{L^\infty}\bigr)\dd s\Bigr)\,,
\end{equation}
for $t \in [0,T]$. Note that~\eqref{g-L1} remains valid in
the general case where $g$ changes its sign.

\medskip\noindent{\bf Step 2:} $L^1$--$L^2$ estimate. By a similar
calculation, we find
\begin{align*}
  \frac12 \frac{\D}{\D t}\int g^2\dd x \,&=\,
  \int g\Bigl(\Delta g + 2\alpha\cdot\nabla g + \alpha^2 g -
  U\cdot\nabla g - (U\cdot\alpha + V)g\Bigr)\dd x \\
  \,&=\, -\int |\nabla g|^2\dd x + \Bigl(\alpha^2 + |\alpha|
  \|U(t)\|_{L^\infty} + \|V(t)\|_{L^\infty}\Bigr)\int g^2\dd x\,.
\end{align*}
To estimate the right-hand side we apply Nash's inequality
\[
  \biggl(\int g^2 \dd x\biggr)^{1+2/n} \,\le\, C_n
  \biggl(\int |g| \dd x\biggr)^{4/n} \int |\nabla g|^2\dd x\,,
\]
which holds for any $g \in L^1(\R^n) \cap H^1(\R^n)$ with
a constant $C_n > 0$ depending only on the space dimension $n$.
We thus obtain the estimate
\begin{equation}\label{g-L2prelim}
  \frac12 \frac{\D}{\D t}\int g^2\dd x \,\le\, - \frac{\Bigl(
  \int g^2 \dd x\Bigr)^{1+2/n}}{C_n \Bigl(\int |g| \dd x\Bigr)^{4/n}}
  + \Bigl(\alpha^2 + |\alpha| \|U(t)\|_{L^\infty} + \|V(t)\|_{L^\infty}
  \Bigr)\int g^2\dd x\,,
\end{equation}
which is a differential inequality for the $L^2$ norm of
the solutions of~\eqref{g-evol}. To solve~\eqref{g-L2prelim},
we temporarily denote
\begin{align*}
  \Lambda(t) \,&=\, \alpha^2 t + \int_0^t \bigl(|\alpha|
  \|U(s)\|_{L^\infty} + \|V(s)\|_{L^\infty}\bigr)\dd s \,, \\
  A(t) \,&=\, \exp(-\Lambda(t)) \int |g(x,t)|\dd x \,\le\, A(0)\,, \\
  B(t) \,&=\, \exp(-2\Lambda(t)) \int g(x,t)^2 \dd x\,,
  \qquad t \in [0,T]\,.
\end{align*}
Here the bound $A(t) \le A(0)$ is a reformulation of~\eqref{g-L1}.
Using~\eqref{g-L2prelim}, we find
\begin{align*}
  B'(t) \,&\le\, -\frac{2}{C_n}\,\frac{\Bigl(\int g^2 \dd x
  \Bigr)^{1+2/n}}{\Bigl(\int |g| \dd x\Bigr)^{4/n}}\,e^{-2\Lambda(t)}
  \,\le\, -\frac{2}{C_n}\,\frac{\Bigl(B(t)\,e^{2\Lambda(t)}
  \Bigr)^{1+2/n}}{\Bigl(A(0)\,e^{\Lambda(t)}\Bigr)^{4/n}}\,e^{-2\Lambda(t)} \\
  \,&=\, -\frac{2}{C_n}\,\frac{B(t)^{1+2/n}}{A(0)^{4/n}}\,,
  \qquad 0 < t \le T\,.
\end{align*}
Integrating this simple differential inequality we obtain $B(t) \le
(C_n' A(0))^2 t^{-n/2}$ for $t \in (0,T]$, where $C_n' = (nC_n/4)^{1/4}$.
In other words, we have proved the $L^1$--$L^2$ estimate
\begin{equation}\label{g-L2}
  \|g(t)\|_{L^2} \,\le\, \frac{C_n'}{t^{n/4}}\,\|g(0)\|_{L^1}
  \,\exp\Bigl(\alpha^2 t + \int_0^t \bigl(|\alpha|
  \|U(s)\|_{L^\infty} + \|V(s)\|_{L^\infty}\bigr)\dd s\Bigr)\,,
\end{equation}
for all $t \in (0,T]$.

\medskip\noindent{\bf Step 3:} $L^1$--$L^\infty$ estimate. We
consider the adjoint equation
\begin{equation}\label{g-adj}
  \partial_t \tilde g - U\cdot\nabla \tilde g + (U\cdot\alpha + V)
  \tilde g \,=\, \Delta \tilde g - 2\alpha\cdot \nabla \tilde g
  + \alpha^2 \tilde g\,,
\end{equation}
which has exactly the same structure as~\eqref{g-evol}. In particular,
the $L^1$--$L^2$ bound~\eqref{g-L2} holds for the solutions of
\eqref{g-adj}, and using a standard duality argument this
implies the following $L^2$--$L^\infty$ estimate for the
solutions of~\eqref{g-evol}:
\begin{equation}\label{g-Linfty}
  \|g(t)\|_{L^\infty} \,\le\, \frac{C_n'}{t^{n/4}}\,\|g(0)\|_{L^2}
  \,\exp\Bigl(\alpha^2 t + \int_0^t \bigl(|\alpha|
  \|U(s)\|_{L^\infty} + \|V(s)\|_{L^\infty}\bigr)\dd s\Bigr)\,,
\end{equation}
for all $t \in (0,T]$. To obtain the $L^1$--$L^\infty$
bound we estimate $\|g(t/2)\|_{L^2}$ in terms of $\|g(0)\|_{L^1}$
using~\eqref{g-L2}, and then $\|g(t)\|_{L^\infty}$ in terms of
$\|g(t/2)\|_{L^2}$ using the analogue of~\eqref{g-Linfty}.
Denoting $C_n'' = 2^{n/2}C_n'^2$, this gives
\begin{align}\nonumber
  \|g(t)\|_{L^\infty} \,&\le\, \frac{C_n''}{t^{n/2}}\,\|g(0)\|_{L^1}
  \,\exp\Bigl(\alpha^2 t + \int_0^t \bigl(|\alpha|
  \|U(s)\|_{L^\infty} + \|V(s)\|_{L^\infty}\bigr)\dd s\Bigr)
    \\ \label{g-L1Linfty}
  \,&\le\, \frac{C_n''}{t^{n/2}}\,\|g(0)\|_{L^1}
  \,\exp\Bigl(\alpha^2 t + 2 K_1 |\alpha| \sqrt{t} + K_2\Bigr)\,,
\end{align}
for all $t \in (0,T]$, where in the second inequality we
used definitions~\eqref{UVbounds}.

\medskip\noindent{\bf Step 4:} conclusion. By construction
the solutions of~\eqref{g-evol} can be represented as
\[
  g(x,t) \,=\, \int e^{\alpha\cdot(y-x)} \Phi(x,t;y)
  g(y,0)\dd y\,, \qquad x \in \R^n\,, \qquad 0 < t \le T\,,
\]
where $\Phi(x,t;y) = \Phi_{U,V,1}(x,t;y,0)$ is the fundamental
solution of equation~\eqref{f-evol} with $\nu = 1$.  Estimate
\eqref{g-L1Linfty}, which holds for all smooth and compactly supported
initial data $g(x,0)$, is thus equivalent to the pointwise upper bound
\begin{equation}\label{Phiupper}
  \Phi(x,t;y) \,\le\, \frac{C_n''}{t^{n/2}}\, e^{\alpha\cdot(x-y)}
  \,\exp\Bigl(\alpha^2 t + 2 K_1 |\alpha| \sqrt{t} + K_2\Bigr)\,,
  \qquad x,y \in \R^n\,, \quad 0 < t \le T\,.
\end{equation}
The vector $\alpha \in \R^n$ was arbitrary, and the dependence
upon $\alpha$ is fully explicit in~\eqref{Phiupper}. Given
$x,y \in \R^n$ and $t > 0$, we can thus choose $\alpha =
-(x-y)/(2t)$, in which case~\eqref{Phiupper} becomes
\begin{equation}\label{Phiupper2}
  \Phi(x,t;y) \,\le\, \frac{C_n''}{t^{n/2}} \,\exp\Bigl(
  -\frac{|x-y|^2}{4t} + K_1 \frac{|x-y|}{\sqrt{t}} + K_2\Bigr)\,.
\end{equation}
This proves~\eqref{fund-bound} for $\nu = 1$ and $s = 0$,
and the general case easily follows. \QED
 

\end{document}